\documentclass[11pt]{article}
\usepackage{mydef2col}
\usepackage{markArticle}

\graphicspath{{figs/}{./}}

\title{Diagonal Frog meets ADI: trading matrix exponentials for rational maps in the Fokker--Planck equation}

\shorttitle{Positivity-preserving ADI schemes}

\author{
\authorstyle{
Andrey Itkin\textsuperscript{1}\thanks{e-mail: \url{aitkin@nyu.edu}} \,
and Rakhymzhan Kazbek\textsuperscript{2}\thanks{e-mail: \url{rakhymzhan.kazbek@umu.se}}
}
\newline\newline
\textsuperscript{1}
\institution{FRE Department, Tandon School of Engineering, New York University, USA.} \\
\textsuperscript{2}
\institution{Department of Mathematics and Mathematical Statistics, UMEA Universitet, Sweden.}
}

\date{\today}

\begin{document}

\maketitle

\lettrineabstract{A companion paper \cite{ItkinDF2026} introduced the Diagonal Frog (DF) positivity-preserving schemes for anisotropic Fokker--Planck equations, advancing each directional substep by a Krylov-computed matrix exponential, which dominates the cost. Replacing that exponential by a rational map $r(\gamma L)$ reduces the substep to a banded solve, but the positivity argument no longer applies. We prove that for eventually exponentially positive generators the entrywise sign of $r(\gamma L)$ at large steps is decided by a single number, the value $r(\infty)$ taken on infinitely stiff modes. Nonnegativity holds above a computable threshold when $0\le r(\infty)<1$, and at most on a bounded interval, empty or vanishingly narrow in all our tests, when $r(\infty)<0$. The criterion rejects the Crank--Nicolson (trapezoidal) method, where $r(\infty)=-1$, and selects the subdiagonal Pad\'e$(0,2)$ method, which is second order, L-stable and provably positive above an explicit threshold. The resulting DF-ADI scheme costs $O(N)$ per step, keeps the implicit factorized mixed derivative unchanged, is second order in space and time, and conserves discrete mass exactly. In the strong cross-diffusion regime, however, the directional factors demand a step larger than the mixed derivative permits, so the composite second-order scheme is only empirically positive there, and the criterion serves to discriminate the well-behaved multiplicative factors from the stabilizing-correction schemes rather than to guarantee positivity. Against the Krylov exponential it runs ten to thirty-two times faster at matched accuracy in our tests with the gain growing with the mesh. We extend the construction to the backward Kolmogorov equation and to jump-diffusion models.
}

\section{Introduction}

The Fokker--Planck equation (FPE) governs the evolution of probability density functions in non-equilibrium stochastic systems \cite{Risken1996,Kwok2018}. Two structural properties of the exact solution, nonnegativity and conservation of total mass, are notoriously fragile under discretization when the diffusion tensor is anisotropic (cross-diffusion) or the dynamics contain jumps. In the companion paper \cite{ItkinDF2026} we introduced the Diagonal Frog (DF) family of finite-difference schemes for this problem. Its three building blocks are: (i) directional one-dimensional operators discretized in divergence form whose matrices are M-matrices or eventually exponentially positive (EM) matrices, so that the matrix exponential $e^{\Delta t L}$ is entrywise nonnegative either for all $\Delta t>0$ or for all $\Delta t$ above an explicit threshold $\tau_0$; (ii) a conservative one-sided discretization of the mixed derivative, advanced by a trapezoidal factor whose implicit half is solved by a factorized Picard iteration at linear cost, following \cite{Itkin3D}; and (iii) a symmetric Strang composition of these substeps, with the directional exponentials applied by polynomial Krylov methods.

The Krylov machinery is the single most expensive ingredient of that construction. Each directional substep costs $O(m^2 N + m^3)$ operations, where $m$ is the Krylov dimension, and although $m$ grows only sub-linearly with the problem size, the constants, the Arnoldi orthogonalization and the stopping criteria for the exponential action add real implementation weight. In the pricing literature the standard answer to this cost is the alternating direction implicit (ADI) paradigm: advance each spatial direction by an implicit one-dimensional solve with banded matrices, and treat the coupling terms by explicit stages or stabilizing corrections. The schemes of Douglas, Craig--Sneyd, modified Craig--Sneyd (MCS) and Hundsdorfer--Verwer (HV) are the canonical representatives, and their unconditional $L_2$-stability in the presence of mixed derivatives was established in \cite{HoutWelfert2007,HoutFoulon2010}. In \cite{Itkin2015Forward} one of these backward schemes was transposed to obtain a forward scheme for the density that is exactly consistent with its backward counterpart.

The purpose of the present paper is not to substitute one time integrator for another. At the level of the scheme the change is a single factor: every Krylov exponential $e^{sA_\alpha}$ of the DF step is replaced by a rational map $r(sA_\alpha)$, while the spatial discretization, the conservative closure, the implicit mixed factor with its factorized Picard solver, the symmetric arrangement of the substeps and the transpose duality are carried over from \cite{ItkinDF2026} unchanged. What the substitution requires is a positivity theory that the companion paper does not contain. Eventual positivity of $e^{tL}$ decides the entrywise sign of one particular map of the generator and says nothing about the others. To choose among rational maps one needs to know how the sign of $r(\gamma L)$ depends on $r$, and that is the question answered here. The criterion is the main theoretical contribution of the paper, the ADI scheme is its first application, and the exponential scheme of \cite{ItkinDF2026} is one member of the family the criterion describes rather than an alternative to it. The design of the scheme then follows from three observations ($\calA$, $\calB$ and $\calC$).

\subsection{Observation $\calA$}

The first observation concerns the directional factors. Replacing the directional exponential by a rational map changes the positivity mechanism in an essential way. Write $L$ for either of the directional matrices $A_x$, $A_y$, and $\gamma$ for the length of the substep. Where the exponential scheme advances the density by $e^{\gamma L}$, a one-step method advances it by $r(\gamma L)$, in which $r$ is the rational function that the method applies in place of the exponential, known as its stability function. Backward Euler gives $r(z)=1/(1-z)$ and the trapezoidal rule gives $r(z)=(1+z/2)/(1-z/2)$. Consistency of order $p$ means $r(z)=e^z+O(z^{p+1})$, so $r(0)=1$ always, and that normalization is what makes each substep conserve discrete mass exactly on a generator with zero column sums.

The number that decides the sign lies at the other end of the range: $r(\infty)=\lim_{|z|\to\infty}r(z)$, the factor by which the method damps infinitely stiff modes. For the first-order upwind (M-matrix) variant of the spatial scheme the backward Euler resolvent $(\mathcal I-\gamma L)^{-1}$ is nonnegative for every $\gamma>0$ and nothing interesting happens. For the second-order DF operators, which are EM but not Metzler, we prove that the resolvent is nonnegative precisely on a one-sided window $\gamma\ge\gamma_0$, a mirror image of the eventual-positivity threshold $\tau_0$ of the exponential: positivity holds for large steps and fails in the refinement limit $\Delta t\to0$.

More generally, for any such $r$, the large-step limit is $r(\gamma L)\to \Pi + r(\infty)(\mathcal I-\Pi)$, where $\Pi$ is the spectral projector onto the stationary density. The entrywise sign of that limit is decided by the single number $r(\infty)$: the limit is strictly positive whenever $0\le r(\infty)<1$, while for $r(\infty)<0$ it has a negative diagonal entry at every node whose stationary mass lies below $|r(\infty)|/(1+|r(\infty)|)$. The trapezoidal (Cayley) map, with $r(\infty)=-1$, therefore admits at most a bounded two-sided window, requiring stationary mass above one half at every node, and in our experiments this window is empty on every mesh in the advection-dominated regime. It opens only on the finest weakly eventually positive mesh we tried, and then to a width of $7\cdot10^{-4}$ in the substep, orders of magnitude below any step one would use. The scheme built on it remains positive in practice on resolved meshes, but carries no usable per-factor guarantee.

The constructive side of the criterion is the central design decision of the paper. The subdiagonal Pad\'e$(0,2)$ approximant $r(z)=1/(1-z+z^2/2)$ is second-order, L-stable and satisfies $r(\infty)=0$, hence enjoys a genuine one-sided positivity window $\gamma\ge\gamma_r$ by the same mechanism as the resolvent, and we adopt it as the flagship directional factor, retaining Crank--Nicolson as the instructive negative case. To our knowledge the connection between eventual positivity of generators \cite{NoutsosTsatsomeros2008,OleskyEtAl2009} and the positivity windows of their rational approximants has not been worked out before.

\subsection{Rational approximations of eventually positive semigroups}

The intersection of rational approximants and eventually positive semigroups bridges infinite-dimensional operator theory and numerical analysis. Because the rigorous mathematical theory of eventually positive semigroups is a relatively recent development, a unified framework explicitly dedicated to their rational approximations remains largely a gap in the published literature. The current research landscape approaches this topic by merging the foundational theory of eventual positivity with numerical positivity-preserving schemes.

\myparagraph{Foundational Theory of Eventually Positive Semigroups.}

The systematic, infinite-dimensional theory of eventually positive $C_0$-semigroups was established to describe systems where positive initial data may lead to sign-changing solutions for small times, but the solutions eventually become and remain strictly positive. \cite{daners2015eventually} formalized this phenomenon by characterizing such semigroups through Perron-Frobenius-type spectral conditions and the resolvent properties of their generators. Subsequent research has focused on the asymptotic behavior and stability of these systems. \cite{arora2021stability} proved that for eventually positive semigroups on spaces of continuous functions, the spectral bound coincides with the growth bound, a result further expanded by \cite{arora2024eventually} for broader spectral analysis. Importantly for numerical approximations, \cite{daners2017towards} demonstrated that eventual positivity is highly sensitive to perturbations; unlike standard positive semigroups, perturbing the generator by a positive operator of large norm can destroy the eventual positivity.

\myparagraph{Rational Approximations and Positivity Preservation.}

Rational approximants to the exponential function (such as Pad\'e approximants, which encompass Crank-Nicolson and Backward Euler schemes) are standard for time-stepping in numerical partial differential equations. A major challenge is designing rational approximations that remain unconditionally positive. While classical numerical analysis establishes conditions for rational functions to preserve standard positivity, \cite{BolleyCrouzeix1978,brenner1979rational}
extending this to complex PDEs remains an active area of development. The core
obstruction is classical. No rational approximation of the exponential that
preserves positivity unconditionally on every dissipative generator can exceed
first order, \cite{BolleyCrouzeix1978}. It is the parabolic counterpart of
Godunov's theorem, and it recurs throughout this paper. For instance,
\cite{guermond2020positive} developed linear rational approximation techniques
that are both positive and asymptotic-preserving in the diffusion limit for
radiation transport equations.

\myparagraph{Synthesis and Literature Gap.}

While extensive literature exists on positivity-preserving rational approximations for standard positive semigroups, the direct study of rational approximations for \emph{eventually} positive semigroups is sparse. The primary mathematical gap lies in translating the theoretical resolvent conditions from \cite{daners2015eventually} into stability criteria for discrete rational functions. Current applied research, particularly in high-order finite-difference schemes for anisotropic Fokker-Planck equations and related quantitative models, frequently encounters matrices that are eventually positive. In these contexts, specific rational approximants (like Backward Euler) are employed empirically to stabilize the asymptotic positivity, but a rigorous, generalized theoretical framework linking the degree of the rational approximant to the eventual positivity time threshold remains an open problem.

\subsection{Observation $\calB$}

The second observation concerns the mixed derivative. Classical ADI schemes treat the mixed term explicitly because an implicit treatment appears to destroy the banded structure that makes the directional solves cheap. This explicit treatment is exactly what the DF analysis identified as the source of positivity loss: the cross stencil is indefinite, the predictor stage $Y_0=(\mathcal I+\Delta t\,F)\bm p^{\,n}$ is never entrywise nonnegative unconditionally, and the $L_2$-stability theory of \cite{HoutWelfert2007}, built on normal commuting operators, says nothing about the sign of the solution.

The factorized Picard solver of \cite{Itkin3D}, however, removes the structural obstruction: it applies the resolvent of the mixed block at $O(N)$ cost through triangular banded sweeps. A fully implicit-in-mixed ADI scheme therefore becomes possible at linear complexity, and we adopt it as the flagship construction. The central factor of the scheme below is identical to the trapezoidal factor of \cite{ItkinDF2026}, and its conditional positivity window transfers verbatim.

It is worth being precise about the reach of that construction, because the same Picard solver would make the mixed block implicit inside a classical stabilizing-correction scheme as well, and one might expect the positivity theory to follow it there. It does not. A Douglas, MCS or Hundsdorfer--Verwer step is not a product of maps of the generators.

Each of its stages has the form
\[
Y_j=(\mathcal I-\theta\Delta t\,F_j)^{-1}[\,Y_{j-1}-\theta\Delta t\,F_j\bm p^{\,n}\,],
\]
and the operand in brackets carries a subtraction. Under the conservative closure $\bm1^\top F_j=0$, so $F_j\bm p^{\,n}$ sums to zero and, unless it vanishes, has entries of both signs. Entrywise nonnegativity of the factor standing in front of the bracket constrains its action on the nonnegative cone only, and is therefore of no use.

Those subtractions are not an artefact of the explicit predictor $Y_0=(\mathcal I+\Delta t\,F)\bm p^{\,n}$. They are what a stabilizing correction is, and they appear at every stage. The dividing line is multiplicative against affine, not implicit against explicit, and not first order against second order in the directional factor. \Cref{ssec:relation} develops the point and records what does transfer to the stabilizing-correction family, namely second order and exact mass conservation.

\subsection{Observation $\calC$}

A third, structural simplification is specific to the FPE on a tensor grid. When the coefficients of the directional operators are separable, the Kronecker lifts $A_x=L_x\otimes I$ and $A_y=I\otimes L_y$ commute exactly, so the directional part of the propagator factorizes without splitting error and the approximate-factorization corrections that motivate the Douglas, MCS and HV stage structures are unnecessary. The only genuinely noncommuting block is the mixed derivative, and it is treated implicitly. In the general (non-separable) case commutation fails, but the symmetric arrangement of the composite step restores second order by the same argument as for Strang splitting.

\vspace{2em}
The main results obtained in this paper are as follows. First, we establish a large-step limit criterion on $r(\infty)$ for EM generators and provide computable threshold estimates (\cref{sec:resolvent}). This analysis yields one-sided resolvent, Pad\'e$(0,2)$, and bounded Cayley positivity windows. It also places the exponential scheme of \cite{ItkinDF2026} as the $r(\infty)=0$ member of a single family.

Second, we introduce the Diagonal Frog Alternating Direction Implicit (DF-ADI) scheme (\cref{sec:dfadi}). This scheme achieves second-order accuracy in space and time with $O(N)$ cost per step. Under the conservative closure, it is exactly mass-conservative for every step size and features a fully characterized conditional positivity regime.

Third, we establish a structural delimitation of the criterion (\cref{ssec:relation}). We demonstrate that the criterion applies to steps that are multiplicative in the generators. In contrast, the affine stage structure of the Douglas, MCS, and HV schemes admits no per-factor sign control regardless of how the mixed derivative is treated. However, second-order accuracy and exact conservation do transfer to these schemes.

Fourth, we extend the approach to the backward Kolmogorov equation and jump-diffusion models (\cref{sec:BKE}). Transpose duality demonstrates that the backward scheme inherits the positivity windows of the forward scheme unchanged. Finally, we present a systematic numerical comparison against the exponential DF scheme and against HV and MCS schemes equipped with the explicit one-sided cross stencil (\cref{sec:numerics}).

One consequence of the criterion should be stated at the outset, because it shapes how the theory is used. In the strong cross-diffusion regime the upper step bound $\Theta$ set by the mixed derivative falls below the lower threshold $2\gamma_r$ set by the directional factors, so for the second-order factors the composite guarantee is vacuous on every mesh we tested; only the first-order backward Euler variant, whose threshold falls with the mesh, recovers a nonempty region there (\cref{crossWindow}). The second-order scheme is nonetheless positive to within $10^{-7}$ in practice (\cref{sec:numerics}). What the theory contributes in that regime is not a guarantee but a discrimination. It explains why the multiplicative Pad\'e$(0,2)$ factors stay well behaved, and why the classical stabilizing-correction schemes, whose stages subtract signed vectors, do not. That distinction is developed in \cref{ssec:relation} and measured in \cref{ssec:cross}.

The remainder of the paper is organized as follows. \Cref{sec:spatial} summarizes the spatial discretization of \cite{ItkinDF2026} required for our analysis. \Cref{sec:resolvent} extends the eventual-positivity theory of the matrix exponential to rational maps for the class of conservative EM generators produced by the DF discretization. \Cref{sec:dfadi} constructs the composite scheme and establishes its order, conservation, and positivity properties. \Cref{sec:BKE} addresses the backward equation and jump processes. \Cref{sec:numerics} presents the numerical experiments, and \cref{sec:conclusion} concludes.

\section{Spatial discretization: a summary} \label{sec:spatial}


We consider the two-dimensional FPE in divergence form,
\begin{equation}   \label{eq:fpe2d}
\frac{\partial p}{\partial t} = -\frac{\partial}{\partial x}[\mu_x p]     -\frac{\partial}{\partial y}[\mu_y p] + \frac{\partial^2}{\partial x^2}[\Sigma_{xx} p] + \frac{\partial^2}{\partial y^2}[\Sigma_{yy} p] + 2\frac{\partial^2}{\partial x\,\partial y}[\Sigma_{xy} p],
\end{equation}
on a tensor grid $\{x_i\}_{i=1}^{N_x}\times\{y_j\}_{j=1}^{N_y}$ with steps $h_x,h_y$, and decompose the right-hand side as
\begin{equation}   \label{eq:split}
\frac{\partial p}{\partial t} = (\mathcal{L}_x + \mathcal{L}_y + \mathcal{L}_{xy})\,p,
\end{equation}
with $\mathcal L_x$, $\mathcal L_y$ the one-dimensional Fokker--Planck operators along the coordinate directions and $\mathcal L_{xy}$ the mixed block.

From \cite{ItkinDF2026} we import the following objects and facts without proof.

\myparagraph{Directional operators.} The 1D discretizations $L_x\in\R^{N_x\times N_x}$, $L_y\in\R^{N_y\times N_y}$ are assembled in discrete flux (divergence) form with a zero-flux boundary closure, so that $\bm 1^\top L_\alpha=0$ on the uniform meshes assumed throughout the analysis. This zero-column-sum property, which we call the conservative closure throughout, is the discrete statement that total probability is exactly conserved. On a graded mesh the identity is replaced by the cell-mass weighted $\bm\Delta^\top L_\alpha=0$, and \cref{ssec:nonuniform} shows that every statement below transfers unchanged once read in the cell-mass variable. Their Kronecker lifts are $A_x=L_x\otimes I_{N_y}$ and $A_y=I_{N_x}\otimes L_y$; when the coefficients of $L_x$ depend only on $x$ and those of $L_y$ only on $y$ these lifts commute, $A_xA_y=A_yA_x$. Depending on the variant of the scheme, $L_\alpha$ is either an M-matrix generator (Metzler, first-order convection treatment) or an EM matrix: $e^{tL_\alpha}\ge0$ for all $t\ge\tau_0^{(\alpha)}$ with an explicit threshold $\tau_0^{(\alpha)}\ge0$ computed from the full matrix including boundary rows. The spectral abscissa of $L_\alpha$ is $0$, the eigenvalue $0$ is simple for irreducible $L_\alpha$, and the associated right eigenvector is the discrete stationary density.

\myparagraph{Mixed operator.} With the separability assumption $2\Sigma_{xy}=\rho\,w_1(x,t)w_2(y,t)$, $w_1,w_2\ge0$, the mixed block is discretized by the conservative one-sided product
\begin{equation} \label{mixProduct}
A_{xy} = \rho\,\mathcal A^{\mathrm F}_{2,x}\mathcal A^{\mathrm B}_{2,y},
\end{equation}
with orientations swapped for $\rho<0$. The operator has zero column sums under the conservative closure, real non-positive spectrum, and is strongly non-normal; no eventual-positivity property holds for it.

\myparagraph{Central factor.} The mixed block is advanced by the trapezoidal factor
\begin{equation}\label{eq:central}
\Phi_{xy}(\Delta t) \;=\;
\Bigl(\mathcal I-\tfrac{\Delta t}{2}A_{xy}\Bigr)^{-1}
\Bigl(\mathcal I+\tfrac{\Delta t}{2}A_{xy}\Bigr),
\end{equation}
whose implicit half is solved by the factorized Picard iteration of \cite{Itkin3D,ItkinDF2026} at $O(N_xN_y)$ cost per iteration. The factor is second-order in time, conserves mass exactly under the conservative closure, and is entrywise nonnegative on a step-size window $\Delta t\le\Theta$ characterized in \cite{ItkinDF2026}; the two coupling variants (first-order Scheme A, second-order Scheme B with the log-Lipschitz resolution condition) transfer unchanged.

\myparagraph{Mass conservation of rational substeps.} For any matrix $M$ with $\bm1^\top M=0$ and any rational $r$ with $r(0)=1$ whose poles avoid $\operatorname{spec}(\Delta t\,M)$, one has $\bm1^\top r(\Delta t\,M)=\bm1^\top$. This corollary of \cite{ItkinDF2026} is used repeatedly below, since every factor of the ADI scheme is of exactly this form.

\section{Positivity of rational maps of EM operators} \label{sec:resolvent}

The exponential DF scheme rests on eventual positivity of $e^{tL}$. The ADI scheme applies rational functions of $L$ instead, and this section extends that theory to them. The extension is not general. It is confined throughout to the class of generators the DF discretization produces: $L\in\R^{n\times n}$ is irreducible, has zero column sums (conservative closure), spectral abscissa $\lambda_1=0$ with simple eigenvalue $0$, and is eventually exponentially positive: $e^{tL}>0$ for all $t\ge\tau_0$. By \cite{NoutsosTsatsomeros2008}, this is equivalent to the existence of a shift $a\ge0$ such that $L+a\mathcal I$ is eventually positive, i.e. possesses the strong Perron--Frobenius property together with its transpose. We write $\Pi=\lim_{t\to\infty}e^{tL}$ for the spectral projector onto the stationary density $\bm p^\infty$ along $\operatorname{ran}L$; entrywise $\Pi_{ij}=p^\infty_i>0$. We also define a function $r$ to be A-acceptability if it places every pole in the open right half-plane.

Throughout this section $\gamma$ denotes the length of a single directional
substep. In the composite scheme of \cref{sec:dfadi} it equals $\Delta t/2$,
so every threshold on $\gamma$ bounds the substep, and the corresponding
composite step is twice that threshold.

\subsection{The backward Euler resolvent} \label{ssec:beresolvent}

\begin{theorem}[Resolvent window] \label{thm:resolvent}
Let $L$ be as above. Then there exists $\gamma_0\in[0,\infty)$ such that
\begin{equation}\label{eq:reswindow}
(\mathcal I-\gamma L)^{-1} > 0 \qquad \text{for all } \gamma \ge \gamma_0 .
\end{equation}
Moreover, if $L$ has at least one negative off-diagonal entry (the EM case proper), then there exists $\gamma_1>0$ such that $(\mathcal I-\gamma L)^{-1}$ has a negative entry for all $\gamma\in(0,\gamma_1)$; in particular $\gamma_0>0$ and positivity fails in the refinement limit $\Delta t\to0$ at fixed mesh. If $L$ is Metzler (the M-matrix variant), then $\gamma_0=0$ and \eqref{eq:reswindow} holds for all $\gamma>0$.
\end{theorem}

\begin{proof}
See \cref{app:resolvent}.
\end{proof}

\begin{myremark}[The Zeno mirror]\label{rem:zenomirror}
\Cref{thm:resolvent} is the exact resolvent analog of the eventual-positivity threshold of the exponential: nonnegativity holds for time steps above a threshold and fails below it, so refining the temporal grid at fixed spatial mesh eventually destroys the sign of the solution. In the exponential case the mechanism was the accumulation of infinitely many small substeps; here it is the dominance of the stationary projector in the Laplace transform of the semigroup near the spectral abscissa. Both thresholds are global properties of the full matrix, boundary rows included.
\end{myremark}

\subsection{The large-step limit and the $r(\infty)$ criterion} \label{ssec:criterion}

Every one-step rational realization of the directional step is of the form $r(\gamma L)$ with $r(0)=1$. The following theorem reduces the large-step entrywise analysis of all such maps to a single number.

\begin{theorem}[Large-step limit]\label{thm:limit}
Let $L$ be as above and let $r$ be an A-acceptable rational function with $r(0)=1$ and finite $r(\infty)$, whose poles avoid $\operatorname{spec}(\gamma L)$ for all $\gamma>0$. Then
\begin{equation}\label{eq:limitmatrix}
\lim_{\gamma\to\infty} r(\gamma L) \;=\; \Pi \;+\; r(\infty)\,(\mathcal I-\Pi),
\end{equation}
with entries $p^\infty_i\,(1-r(\infty))$ off the diagonal and $p^\infty_i + r(\infty)\,(1-p^\infty_i)$ on it.
\end{theorem}

\begin{proof}
See \cref{app:limit}.
\end{proof}

\begin{corollary}[The $r(\infty)$ criterion]\label{cor:rinf}
Under the assumptions of \cref{thm:limit}:
\begin{enumerate}
\item[\textup{(a)}] If $0\le r(\infty)<1$, the limit \eqref{eq:limitmatrix} is strictly positive, and there exists $\gamma_r<\infty$ such that $r(\gamma L)>0$ for all $\gamma\ge\gamma_r$.
\item[\textup{(b)}] If $r(\infty)<0$, the limit has a negative diagonal entry at every node with $p^\infty_i<|r(\infty)|/(1+|r(\infty)|)$; on any mesh containing such a node, nonnegativity of $r(\gamma L)$ can hold at most on a bounded window.
\item[\textup{(c)}] Assume in addition that $r'(0)=1$, the first-order consistency condition. If $L$ has a negative off-diagonal entry, then $r(\gamma L)$ has a negative entry for all sufficiently small $\gamma>0$, and every window is bounded away from zero.
\end{enumerate}
\end{corollary}

\begin{proof}
See \cref{app:rinf}.
\end{proof}

The backward Euler resolvent is the case $r(\infty)=0$ of part (a), recovering \cref{thm:resolvent}. The two subsections below specialize the criterion to the two second-order candidates.

The criterion separates two properties that are easily conflated, and it is worth recording the
separation as a statement in its own right.

\begin{corollary}[Order is not the obstruction]\label{cor:price}
Let $L$ and $r$ satisfy the assumptions of \cref{thm:limit}; in particular $r$ is
A-acceptable with $r(0)=1$ and finite $r(\infty)$. The order of
accuracy of $r$ places no obstruction on the entrywise sign of $r(\gamma L)$ in the large-step
limit; by \cref{thm:limit} that sign is decided by the single number $r(\infty)$. Consequently
two maps of the same order may lie on opposite sides of the criterion: the Cayley map, with
$r(\infty)=-1$, falls under \cref{cor:rinf}(b) and admits at most a bounded window, while the
Pad\'e$(0,2)$ map, equally of order two but with $r(\infty)=0$, falls under \cref{cor:rinf}(a)
and admits a one-sided window. The exponential is the limiting member of the same class,
behaving as $r(\infty)=0$.
\end{corollary}

\begin{proof}
Immediate from \cref{thm:limit} and \cref{cor:rinf}, since the limit \eqref{eq:limitmatrix}
depends on $r$ only through $r(\infty)$, and the order conditions constrain the Taylor
coefficients of $r$ at the origin, which do not enter that limit.
\end{proof}

What Crank--Nicolson pays for is therefore not its second order but its negative value at
infinity, a point we return to quantitatively in \cref{rem:whyexp} and \cref{stiffRing}.

\subsection{The trapezoidal (Cayley) map} \label{ssec:cayley}

The Cayley (Crank--Nicolson) factor has $r(z)=(1+z/2)/(1-z/2)$ with $r(\infty)=-1$, the negative case of \cref{cor:rinf}.

Three names circulate for objects built from this rational function, and we fix them here once.
\emph{Cayley} refers to the map $z\mapsto(1+z/2)/(1-z/2)$ itself, and to windows and thresholds
of that map. \emph{Crank--Nicolson} names a scheme whose time stepping is this map, and it is
the label used for the directional realization of \eqref{eq:dfadi} and in the tables of
\cref{sec:numerics}. \emph{Trapezoidal} is used as the descriptor of the underlying time
discretization, so that the two Crank--Nicolson schemes compared in \cref{sec:numerics} can be
referred to jointly as the trapezoidal ones, and so that the central factor \eqref{eq:central},
which applies the same rule to $A_{xy}$ but is one factor of a scheme rather than a scheme, can
be named without suggesting that the whole composite step is Crank--Nicolson.

The identity
\begin{equation}\label{eq:cayleyident}
r(\gamma L) \;=\; -\,\mathcal I \;+\; 2\,\Bigl(\mathcal I-\tfrac{\gamma}{2}L\Bigr)^{-1}
\end{equation}
reduces its entrywise analysis to that of the resolvent. Off the diagonal, $r(\gamma L)$ and the resolvent at shift $\gamma/2$ have entries of the same sign. On the diagonal, nonnegativity of $r(\gamma L)$ requires the resolvent diagonal to be at least $1/2$, and this is where the large-step behavior departs from \cref{thm:resolvent}.

\begin{theorem}[Cayley window]\label{thm:cayley}
Let $L$ be as above and let $\bm p^\infty$ be its stationary density. Then:
\begin{enumerate}
\item[\textup{(a)}] $r(\gamma L)\to -\mathcal I+2\Pi$ entrywise as $\gamma\to\infty$; hence, if $p^\infty_i<1/2$ for some node $i$, there exists $\gamma_+<\infty$ such that $r(\gamma L)$ has a negative diagonal entry for all $\gamma>\gamma_+$.
\item[\textup{(b)}] If $L$ has a negative off-diagonal entry, there exists $\gamma_->0$ such that $r(\gamma L)$ has a negative entry for all $\gamma\in(0,\gamma_-)$.
\item[\textup{(c)}] Nonnegativity of $r(\gamma L)$ can therefore hold at most on a bounded window $\gamma\in[\gamma_-,\gamma_+]$. The window is nonempty whenever the resolvent window of \cref{thm:resolvent} at shift $\gamma/2$ intersects the set $\{\gamma:\operatorname{diag}(\mathcal I-\tfrac{\gamma}{2}L)^{-1}\ge\tfrac12\}$.
\end{enumerate}
\end{theorem}

\begin{proof}
See \cref{app:cayley}.
\end{proof}

\begin{myremark}[Empirical emptiness of the Cayley window]\label{rem:cnempty}
In all our experiments the window of \cref{thm:cayley}(c) is empty: for the EM directional operators, at the step where the off-diagonal entries of the resolvent first turn nonnegative, the resolvent diagonal has already descended below $1/2$, so the lower and upper failure mechanisms overlap on every mesh tested (\cref{sec:numerics}). We have no proof that the window is empty for every generator in the class of \cref{sec:resolvent}. The emptiness is an empirical finding, stated here as a conjecture supported by the thresholds of \cref{ouThresh} and the sweeps of \cref{sec:numerics}. The Crank--Nicolson realization of the directional factors therefore carries no per-factor positivity guarantee. It remains positive in practice on resolved meshes, and we report it in that capacity, but the guaranteed second-order factor of this paper is the Pad\'e$(0,2)$ map of \cref{ssec:pade}.
\end{myremark}

\begin{myremark}[The exponential and the sign of $r(\infty)$]\label{rem:whyexp}
For the exponential, positivity above the threshold $\tau_0$ persists for all larger steps, because $e^{tL}\to\Pi>0$; in the language of \cref{thm:limit} the exponential behaves as the case $r(\infty)=0$. The price paid by Crank--Nicolson is therefore not the price of second order as such but of its negative $r(\infty)$: the Pad\'e$(0,2)$ map below is second-order and recovers the one-sided window. The backward Euler factor, with $r(\infty)=0$, retains the one-sided window of \cref{thm:resolvent} at the cost of first-order accuracy, mirroring the fallback role it plays in \cite{ItkinDF2026}.
\end{myremark}

\subsection{The Pad\'e$(0,2)$ factor} \label{ssec:pade}

The subdiagonal Pad\'e$(0,2)$ approximant
\begin{equation}\label{eq:pade02}
r_{02}(z) \;=\; \frac{1}{1-z+z^2/2}
\end{equation}
satisfies $r_{02}(z)=e^z+O(z^3)$, so it is second-order in time, and it is L-stable: $|r_{02}(\mathrm{i}y)|^2=1/(1+y^4/4)\le1$ on the imaginary axis, its poles $1\pm\mathrm{i}$ lie in the right half-plane, and $r_{02}(\infty)=0$.

\begin{theorem}[Pad\'e window]\label{thm:pade}
Let $L$ be as above. Then there exists $\gamma_r\in[0,\infty)$ such that
\begin{equation}\label{eq:padewindow}
r_{02}(\gamma L) \;=\; \Bigl(\mathcal I-\gamma L+\tfrac{\gamma^2}{2}L^2\Bigr)^{-1} \;>\; 0
\qquad \text{for all } \gamma\ge\gamma_r ,
\end{equation}
and, if $L$ has a negative off-diagonal entry, $r_{02}(\gamma L)$ has a negative entry for all sufficiently small $\gamma>0$, so $\gamma_r>0$ in the EM case proper.
\end{theorem}

\begin{proof}
See \cref{app:pade}.
\end{proof}

\begin{myremark}[The Metzler case]\label{rem:pademetzler}
Unlike the resolvent, the Pad\'e$(0,2)$ map need not be unconditionally positive when $L$ is Metzler. For the resolvent, \cref{thm:resolvent} gives $\gamma_0=0$ in the M-matrix variant, because $\mathcal I-\gamma L$ is itself an M-matrix and its inverse is entrywise nonnegative for every $\gamma>0$. The second-order map loses this structure. For a tridiagonal Metzler $L$ the distance-two entries of the matrix in \eqref{eq:padewindow} are
\begin{equation}\label{eq:dist2}
\Bigl(\mathcal I-\gamma L+\tfrac{\gamma^2}{2}L^2\Bigr)_{i,i\pm2}
=\tfrac{\gamma^2}{2}\bigl(L^2\bigr)_{i,i\pm2}
=\tfrac{\gamma^2}{2}\,L_{i,i\pm1}L_{i\pm1,i\pm2}\;\ge\;0 ,
\end{equation}
a product of two nonnegative off-diagonals, strictly positive wherever consecutive couplings do not vanish. The matrix therefore has positive off-diagonal entries, is not an M-matrix, and its inverse is no longer sign-definite, so the argument that forces $\gamma_0=0$ does not carry over. Hence $\gamma_r$ may be strictly positive even for Metzler $L$, and the conclusion $\gamma_r>0$ is not confined to the EM case proper of \cref{thm:pade}. The pure-diffusion stiff test of \cref{stiffRing} shows this directly: the operator is Metzler, yet the second-order factor undershoots mildly at a stiff step, while the backward Euler factor, first order and true  M-matrix based, stays nonnegative. This is the discrete face of the Bolley–Crouzeix barrier \cite{BolleyCrouzeix1978}: no rational approximation of the exponential of order two or higher preserves positivity unconditionally on every M-matrix.
\end{myremark}

\begin{myremark}[Implementation]\label{rem:padeimpl}
The factor is applied by solving the single linear system $(\mathcal I-\gamma L+\tfrac{\gamma^2}{2}L^2)\,\bm p=\bm b$, banded with twice the bandwidth of $L$, or, via partial fractions over the conjugate pole pair $1\pm\mathrm{i}$ of $1-z+z^2/2$, as
\begin{equation*}
r_{02}(\gamma L)\,\bm b \;=\; 2\,\Ree\Bigl[\,\mathrm{i}\,\bigl((1+\mathrm{i})\mathcal I-\gamma L\bigr)^{-1}\bm b\Bigr],
\end{equation*}
a single complex banded solve with the same bandwidth as $L$. Either route is $O(n)$ per line. Mass is conserved exactly under the conservative closure since $r_{02}(0)=1$ (the mass corollary of \cref{sec:spatial}), and L-stability removes the temporal ringing of Crank--Nicolson in stiff and advection-dominated regimes.
\end{myremark}

\subsection{Quantitative thresholds} \label{ssec:quantitative}

The thresholds $\gamma_0$, $\gamma_r$, $\gamma_-$ and $\gamma_+$ are global
properties of the full matrix $L$ (boundary rows included).  Once $L$ is
assembled they can be located by bisection on the smallest entry of the corresponding matrix function, as was done in \cref{rem:certification}. Each bisection step evaluates one entrywise minimum of a matrix function of $L$, at the cost of a single factorized solve or matrix action, so locating a threshold to working precision takes a few tens of such evaluations and is negligible beside the $T/\Delta t$ steps of a full integration.  It is
useful to have \emph{a priori} upper bounds that can be evaluated from spectral
data of $L$ alone.  Such bounds follow from the decompositions used in the
proofs of \cref{thm:resolvent,thm:limit}.

\myparagraph{Resolvent bound.}
Write $(s\mathcal I-L)^{-1}=s^{-1}\Pi+R(s)$ with $s=1/\gamma$ and
$R(s)=(s\mathcal I-L)^{-1}(\mathcal I-\Pi)$ the reduced resolvent on
$\operatorname{ran}L$.  Let $\delta>0$ be the spectral gap and let
\begin{equation}\label{eq:Cs}
C_\star=\sup_{s\in(0,\delta/2]}\|R(s)\|_{\max}.
\end{equation}
Because $R(s)$ is analytic on $[0,\delta/2]$, $C_\star$ is finite.  From
$[(s\mathcal I-L)^{-1}]_{ij}=s^{-1}p^\infty_i+R_{ij}(s)$ we obtain
\begin{equation}\label{eq:resolvent-bound}
\gamma_0^{\mathrm{bd}}:=\max\Bigl\{\,\frac{2}{\delta},\ \frac{C_\star}{\min_kp^\infty_k}\,\Bigr\}\;\ge\;\gamma_0 .
\end{equation}
The estimate is conservative: it uses a single global constant $C_\star$ and
the smallest stationary mass.  Tighter row-specific bounds are obtained by
replacing $C_\star$ with $-\min_j R_{ij}(s)$ for each row $i$.

\myparagraph{Pad\'e$(0,2)$ bound.}
The decomposition $r_{02}(\gamma L)=\Pi+r_{02}(\gamma L)(\mathcal I-\Pi)$ gives,
on $\operatorname{ran}L$,
\begin{equation}
\|r_{02}(\gamma L)(\mathcal I-\Pi)\|_{\max}\le
\frac{2K}{(\delta\gamma)^2}\quad\text{for }\gamma\ge\gamma^\ddagger,
\end{equation}
where $K$ is a condition number of the eigenvector matrix of
$L|_{\operatorname{ran}L}$ and $\gamma^\ddagger$ is any value with
$\delta \gamma^\ddagger>\sqrt{2}$ (the modulus of the poles $1\pm\mathrm i$).
Hence
\begin{equation}\label{eq:pade-bound}
\gamma_r^{\mathrm{bd}}:=\max\Bigl\{\gamma^\ddagger,\;
\frac{1}{\delta}\sqrt{\frac{2K}{\min_kp^\infty_k}}\Bigr\}\;\ge\;\gamma_r .
\end{equation}

Because the error on $\operatorname{ran}L$ decays as $\gamma^{-2}$, the Pad\'e
threshold is determined by the same spectral data as the resolvent but with a
quadratic dependence on the gap.

\myparagraph{Ornstein--Uhlenbeck benchmark.}
For the advection-dominated OU operator of \cref{ssec:ou} ($\kappa=6$,
$\sigma=0.5$, conservative closure), the quantities $C_\star$, $\delta$, and $K$
are evaluated directly from the assembled matrices. The bounds
\cref{eq:resolvent-bound,eq:pade-bound} overestimate the exact thresholds by
factors that remain stable under refinement: $\gamma_0^{\mathrm{bd}}/\gamma_0$ is
$O(10^2)$ and grows slowly with $n$, whereas $\gamma_r^{\mathrm{bd}}/\gamma_r$
remains between $2.8$ and $2.9$ on every mesh. The latter ratio is
mesh-independent because, in the large--$\gamma$ limit, both the exponential and
the Pad'e map approach the same projector $\Pi$; hence, their thresholds inherit
the same spectral length scale. By contrast, $\gamma_0$ tends to zero under
refinement, while $\gamma_r$ and $\tau_0$ remain $O(1)$.

This contrast should not be interpreted as a defect of the backward Euler factor;
it is important to be explicit about the direction of the inequality. A smaller
$\gamma_0$ \emph{widens} the one-sided window $[2\gamma_0,\infty)$ and therefore
makes the directional condition easier, not harder, to satisfy. The vanishing
threshold reflects the order of the map, rather than the quality of its
positivity window. The resolvent is first-order, and its threshold is governed by
the $O(\gamma)$ term in the Neumann expansion, which vanishes with the mesh. By
contrast, the second-order maps inherit the spectral length scale of the
projector $\Pi$ and therefore retain an $O(1)$ threshold. The price of the wider
window is thus paid in accuracy, not in positivity.

For the composite step, the wider window is not merely formal. The upper endpoint
$\Theta$ in \cref{prop:posadi} is a resolution condition on the cross stencil and
\emph{increases} under refinement, whereas $2\gamma_0$ decreases like $h$. Thus,
the two endpoints move toward one another from opposite sides and eventually
cross: on the meshes of \cref{crossWindow}, the interval $[2\gamma_0,\Theta]$ is
empty on the coarse meshes but nonempty on the fine ones. The backward Euler
variant is therefore the one for which the composite positivity guarantee is
actually available, at the expected cost of first-order accuracy.

The admissible step region of the exponential DF scheme is $[2\tau_0,\Theta]$;
that of the DF-ADI flagship is $[2\gamma_r,\Theta]$.  Because $\gamma_r$ and
$\tau_0$ are of the same order (see \cref{ouThresh}), the two regions have the
same shape and differ only by a constant factor in the lower endpoint.  The backward-Euler fallback $[2\gamma_0,\infty)$ is drawn in \cref{Fwindows}. Its lower endpoint $2\gamma_0$ moves toward zero under refinement, so its guaranteed window widens on fine meshes rather than shrinking. This is the low-order behavior of the resolvent (\cref{thm:resolvent}), in contrast to the mesh-stable $\gamma_r$ of the flagship.

\Cref{tab:maps} collects the four directional maps and their positivity windows.
\begin{table}[!htb]
\centering
\scalebox{0.9}{
\begin{tabular}{|l|c|c|l|l|}
\toprule
\rowcolor[rgb]{ .792,  .929,  .984}
\textbf{Method} & \textbf{Order} & $\bm{r(\infty)}$ & \textbf{Positivity window} & \textbf{Cost per direction} \\
\hline
Backward Euler & $1$ & $0$ & $[\gamma_0,\infty)$, $\gamma_0\!\to\!0$ & one banded solve \\
\hline
Crank--Nicolson & $2$ & $-1$ & bounded, empirically empty & one banded solve \\
\hline
\rowcolor[rgb]{ .557,  .851,  .451}
Pad\'e$(0,2)$ & $2$ & $0$ & $[\gamma_r,\infty)$, $\gamma_r=O(1)$ & one complex banded solve \\
\hline
Exponential & $\infty$ & $0$ & $[\tau_0,\infty)$ & Krylov, $O(m^2N)$ \\
\bottomrule
\end{tabular}}
\caption{The directional maps of \cref{sec:resolvent} at a glance, on the substep of length $\gamma=\Delta t/2$; orders are temporal. The windows are those of \cref{thm:resolvent} (backward Euler), \cref{thm:cayley} (Crank--Nicolson, see \cref{rem:cnempty}) and \cref{thm:pade} (Pad\'e), and $\tau_0$ is the eventual-positivity threshold of the exponential. Only the maps with $r(\infty)=0$ carry a one-sided window, and the flagship Pad\'e$(0,2)$ is the second-order member of that set.}
\label{tab:maps}%
\end{table}%

\subsection{Non-uniform meshes} \label{ssec:nonuniform}

Everything above is stated for a fixed generator and does not require a uniform
mesh. The conservative closure, however, requires some care when the spacing
varies, and it is useful to distinguish what changes from what does not.

Let $x_1<\dots<x_n$ be arbitrary, place the faces at the midpoints
$x_{i+1/2}=(x_i+x_{i+1})/2$, and let $\Delta_i$ denote the dual-cell widths,
$\Delta_i=(x_{i+1}-x_{i-1})/2$ in the interior. Writing the update in flux form,
$\dot{p}*i=-(J*{i+1/2}-J_{i-1/2})/\Delta_i$, with zero flux through the two
boundary faces, gives by telescoping
\begin{equation}\label{eq:nuconserve}
\sum_i \Delta_i,\dot p_i
= -\bigl(J_{n+1/2}-J_{1/2}\bigr)=0,
\qquad\text{that is}\qquad
\bm\Delta^\top L=0,
\end{equation}
where $\bm\Delta=(\Delta_1,\dots,\Delta_n)^\top$. Thus, the conserved functional
is the cell mass $\sum_i\Delta_i p_i$, and the column-sum identity is
\emph{weighted}. In general, it is not $\bm1^\top L=0$, which is the hypothesis
used in \cref{sec:resolvent}. On a uniform mesh, $\bm\Delta$ is constant and the
two statements coincide, explaining why this distinction does not arise elsewhere
in the paper.

The theory nevertheless transfers directly. Introduce the cell-mass variable $\bm
m=D_\Delta\bm p$, where $D_\Delta=\operatorname{diag}(\bm\Delta)$. Then $\dot{\bm
m}=\tilde L\bm m$ with $\tilde L=D_\Delta L D_\Delta^{-1}$, and
\begin{equation}
\bm1^\top\tilde L = (\bm \Delta^\top L) D_\Delta^{-1} = 0
\end{equation}
by \eqref{eq:nuconserve}. Since $D_\Delta$ is a \emph{positive} diagonal matrix,
this similarity transformation preserves the entrywise sign pattern and hence the
M-matrix/EM structure of the generator. Moreover, for any rational function $r$,
\begin{equation}\label{eq:nusimilar}
r(\gamma\tilde L) = D_\Delta,r(\gamma L),D_\Delta^{-1},
\end{equation}
so $r(\gamma\tilde L)$ is entrywise nonnegative precisely when $r(\gamma L)$ is.
The thresholds $\tau_0$, $\gamma_0$, and $\gamma_r$ are therefore invariant under
this change of variables. Consequently, all results of \cref{sec:resolvent},
including \cref{thm:limit}, \cref{cor:rinf}, and \cref{thm:pade}, carry over
verbatim to a non-uniform mesh when expressed in the cell-mass variable.

What does change with the mesh is the size of the thresholds, and the three maps
behave quite differently. \Cref{nuThresh} sweeps a $\sinh$-stretched grid at
fixed $n$, clustering nodes near the centre of the well, and reports the
thresholds together with the weighted column sum and the locations of the
negative off-diagonal entries. Two features stand out. First, the EM structure is
preserved: the negative off-diagonal entries remain on the bands $j-i=\pm2$ for
every stretching, so the operator remains EM rather than becoming Metzler, while
\eqref{eq:nuconserve} holds to round-off throughout. Second, the thresholds
separate markedly. Over a stretching range of $h_{\max}/h_{\min}=8.5$, the
eventual-positivity threshold $\tau_0$ of the exponential increases by a factor
of $2.4$ and the backward Euler threshold $\gamma_0$ by a factor of $4.2$,
whereas $\gamma_r$ changes by only $0.6\%$. Thus, the Pad'e$(0,2)$ threshold is
insensitive not only to refinement, as shown in \cref{ouThresh}, but also to mesh
grading—a useful property for a threshold that must be certified once for a given
problem.
\begin{table}[htbp]
\centering
\scalebox{0.85}{
\begin{tabular}{|l|c|c|r|r|r|c|c|}
\toprule
\rowcolor[rgb]{ .792,  .929,  .984}
\textbf{grid} & $h_{\max}/h_{\min}$ & $h_{\min}$ &
\multicolumn{1}{c|}{$\tau_0$} & \multicolumn{1}{c|}{$\gamma_0$} &
\multicolumn{1}{c|}{$\gamma_r$} &
$\max|\bm\Delta^\top L|$ & \textbf{negative bands} \\
\hline
uniform     & 1.0 & 0.07500 & 4.371e-01 & 8.093e-03 & 1.570e+00 & 7e-15 & $\pm2$ \\
\hline
$c=4.0$     & 1.8 & 0.05974 & 5.513e-01 & 1.334e-02 & 1.574e+00 & 7e-15 & $\pm2$ \\
\hline
$c=2.0$     & 3.1 & 0.04547 & 7.248e-01 & 2.005e-02 & 1.577e+00 & 4e-15 & $\pm2$ \\
\hline
$c=1.5$     & 4.1 & 0.03928 & 8.104e-01 & 2.361e-02 & 1.577e+00 & 7e-15 & $\pm2$ \\
\hline
$c=1.0$     & 6.0 & 0.03115 & 9.305e-01 & 2.911e-02 & 1.579e+00 & 1e-14 & $\pm2$ \\
\hline
\rowcolor[rgb]{ .557,  .851,  .451}
$c=0.7$     & 8.5 & 0.02490 & 1.030e+00 & 3.423e-02 & 1.580e+00 & 1e-14 & $\pm2$ \\
\bottomrule
\end{tabular}}
\caption{Positivity thresholds on a $\sinh$-stretched mesh, $x=c\sinh\xi$ with
$\xi$ uniform, for the Ornstein--Uhlenbeck operator \eqref{eq:ou} with
$\kappa=\sigma=1$ on $[-6,6]$ at fixed $n=161$; smaller $c$ gives stronger
grading. The weighted column sum \eqref{eq:nuconserve} holds to round-off at
every stretching, and the negative off-diagonal entries remain confined to the
bands $\pm2$, so the operator remains EM. Across the sweep, $\tau_0$ increases
by a factor of $2.4$ and $\gamma_0$ by $4.2$, whereas $\gamma_r$ changes by
only $0.6\%$.}
\label{nuThresh}%
\end{table}%

\section{The DF-ADI scheme for the 2D Fokker--Planck equation} \label{sec:dfadi}

With the theoretical foundation for the rational maps established, we now construct the full DF-ADI scheme. This composite numerical step replaces the computationally expensive matrix exponentials with efficiently solvable Pad\'e$(0,2)$ rational factors. The resulting architecture achieves second-order accuracy in space and time at an $O(N)$ computational cost per step, ensures exact discrete mass conservation under the conservative closure, and features a fully characterized conditional positivity regime.

\myparagraph{Directional factors and commutation.}
Let $\Delta t$ denote the uniform temporal step size. Following the methodology of \cite{ItkinDF2026}, we freeze the spatial operators' coefficients at the temporal midpoint $t_n + \Delta t/2$. The integration over a directional half-step of length $s$, where $s=\Delta t/2$ is the substep denoted $\gamma$ in \cref{sec:resolvent}, is realized by applying the subdiagonal Pad\'e$(0,2)$ approximant, resulting in the factors:
\begin{equation}\label{eq:dirfactor}
\Phi_\alpha(s) \;=\; r_{02}(sA_\alpha) \;=\; \Bigl(\mathcal I-sA_\alpha+\tfrac{s^2}{2}A_\alpha^2\Bigr)^{-1},
\qquad \alpha\in\{x,y\}.
\end{equation}
Operationally, each factor requires $N_y$ (respectively, $N_x$) independent banded linear solves of dimension $N_x$ (respectively, $N_y$) at a total cost of strictly $O(N_x N_y)$, implemented per \cref{rem:padeimpl}. The Crank--Nicolson realization, $\bigl(\mathcal I-\tfrac{s}{2}A_\alpha\bigr)^{-1}\bigl(\mathcal I+\tfrac{s}{2}A_\alpha\bigr)$, is retained as a comparative baseline, as it lacks a per-factor positivity guarantee (\cref{rem:cnempty}).

The consecutive application $\Phi_x\Phi_y$ constitutes an alternating-direction sweep for the uncoupled directional generator $A_x+A_y$, structurally similar to the Peaceman--Rachford method but utilizing fully implicit factors. Under conditions where the coefficients are strictly separable, the Kronecker lifts commute exactly, yielding:
\begin{equation}\label{eq:commute}
\Phi_x(s)\,\Phi_y(s) \;=\; e^{s(A_x+A_y)} + O(s^3).
\end{equation}
This factorization introduces no structural splitting error between the directions beyond the local truncation error of the rational approximation itself. In the non-separable case, precise commutation fails, but the subsequent symmetric composite step restores overall second-order accuracy through the standard Strang splitting argument.

\myparagraph{The composite step.}
The scheme advances the solution by
\begin{equation}\label{eq:dfadi}
\bm p^{\,n+1}
= \Phi_x\!\Bigl(\tfrac{\Delta t}{2}\Bigr)\,
  \Phi_y\!\Bigl(\tfrac{\Delta t}{2}\Bigr)\,
  \Phi_{xy}(\Delta t)\,
  \Phi_y\!\Bigl(\tfrac{\Delta t}{2}\Bigr)\,
  \Phi_x\!\Bigl(\tfrac{\Delta t}{2}\Bigr)\,
  \bm p^{\,n},
\end{equation}
with the central factor \eqref{eq:central} applied by the factorized Picard solver. The structure is that of the exponential DF step with every Krylov exponential replaced by a banded rational solve; with the Pad\'e$(0,2)$ factors the directional stages contain no explicit half at all, and no stage of the scheme touches the cross stencil explicitly except the explicit half of the central factor, exactly as in \cite{ItkinDF2026}.

\Cref{Fstrang} shows the sequence of substeps. The operator product in \eqref{eq:dfadi} is read
from right to left, as usual, while the data flow in the diagram runs from left to right; the
composition is symmetric, so the two orderings coincide.

\begin{figure}[!htb]
\centering
\begin{tikzpicture}[
  >=stealth,
  box/.style={draw, rounded corners=2pt, minimum height=10mm, minimum width=21mm,
              align=center, font=\small},
  dirn/.style={box, fill=blue!8},
  mixn/.style={box, fill=green!16},
  cost/.style={font=\scriptsize, align=center, text width=25mm}
]
\node (pn) {$\bm p^{\,n}$};
\node[dirn, right=5mm of pn]  (x1) {$\Phi_x(\tfrac{\Delta t}{2})$};
\node[dirn, right=3mm of x1]  (y1) {$\Phi_y(\tfrac{\Delta t}{2})$};
\node[mixn, right=3mm of y1]  (xy) {$\Phi_{xy}(\Delta t)$};
\node[dirn, right=3mm of xy]  (y2) {$\Phi_y(\tfrac{\Delta t}{2})$};
\node[dirn, right=3mm of y2]  (x2) {$\Phi_x(\tfrac{\Delta t}{2})$};
\node[right=5mm of x2] (pn1) {$\bm p^{\,n+1}$};

\draw[->] (pn) -- (x1);  \draw[->] (x1) -- (y1);  \draw[->] (y1) -- (xy);
\draw[->] (xy) -- (y2);  \draw[->] (y2) -- (x2);  \draw[->] (x2) -- (pn1);

\node[cost, below=2.5mm of x1] {$N_y$ banded solves\\ of size $N_x$};
\node[cost, below=2.5mm of y1] {$N_x$ banded solves\\ of size $N_y$};
\node[cost, below=2.5mm of xy] (xyc) {factorized Picard,\\ $O(N_xN_y)$ per sweep};
\node[cost, below=2.5mm of y2] {$N_x$ banded solves\\ of size $N_y$};
\node[cost, below=2.5mm of x2] {$N_y$ banded solves\\ of size $N_x$};

\draw[dashed, gray!70] ($(xy.north)+(0,3mm)$) -- (xy.north);
\draw[dashed, gray!70] (xyc.south) -- ($(xyc.south)+(0,-3mm)$);
\node[font=\scriptsize, gray!70, above=3.5mm of xy] {mirror};
\end{tikzpicture}
\caption{One step of the DF-ADI scheme \eqref{eq:dfadi}. The two shaded families are the
directional factors, each a set of independent one-dimensional banded solves, and the implicit
mixed factor, applied by the factorized Picard iteration. The half-steps are mirrored about the
central factor, which is what makes the composition symmetric and therefore second order in
time. No stage touches the cross stencil explicitly except the explicit half of
\eqref{eq:central}, and with the Pad\'e$(0,2)$ factors the directional stages have no explicit
half at all.}
\label{Fstrang}
\end{figure}
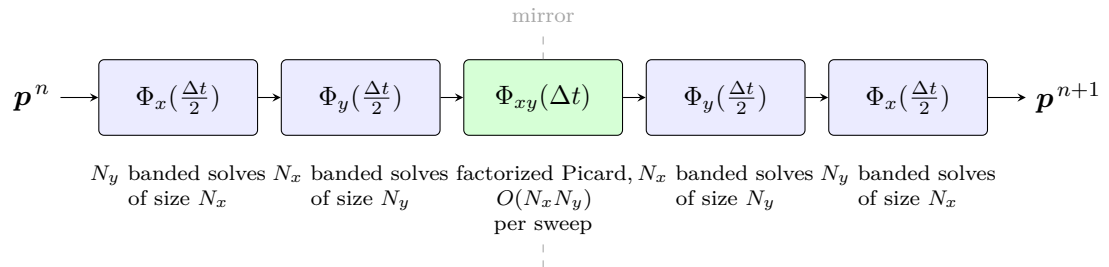

\begin{proposition}[Order]\label{prop:order}
Each factor of \eqref{eq:dfadi} approximates the corresponding exponential to $O(\Delta t^3)$, hence the composite step reproduces the symmetric Strang product of \cite{ItkinDF2026} to $O(\Delta t^3)$ per step, and the scheme is second-order in time. With the Scheme B coupling in the central factor the spatial defect is $O(\max(h_x^2,h_y^2))$, and the assembled scheme converges at second order under the joint refinement $\Delta t\sim h$.
\end{proposition}

\begin{proof}
See \cref{app:order}.
\end{proof}

\subsection{Relation to classical ADI schemes} \label{ssec:relation}

The scheme \eqref{eq:dfadi} is an alternating direction implicit method in the sense of Peaceman--Rachford. Each direction is advanced by an implicit one-dimensional solve with banded matrices, and the directions alternate within the step. It is not a member of the stabilizing-correction family of Douglas, Craig--Sneyd, MCS and Hundsdorfer--Verwer, and since those schemes are the standard reference point in the pricing literature the difference is worth stating precisely.

The stage structures of that family exist to repair two approximations. The first is the approximate factorization of noncommuting directional blocks. The second is the explicit treatment of the mixed derivative, which is forced by the absence of a cheap solver for the cross block. Neither repair is needed here. In the separable case the Kronecker lifts commute exactly, so by \eqref{eq:commute} the directional factorization carries no splitting error, and in the non-separable case the symmetric arrangement of \eqref{eq:dfadi} restores second order without a correction stage. The mixed block is solved implicitly at $O(N)$ cost by the factorized Picard iteration, so it never enters a stage explicitly except through the trapezoidal half of \eqref{eq:central}. The stability theory of the explicit-mixed alternatives is developed in \cite{HoutWelfert2007,HoutFoulon2010}, and we use those schemes as comparators in \cref{sec:numerics}.

The same solver would allow the mixed block to be made implicit inside a stabilizing-correction scheme, by promoting $A_{xy}$ from the explicitly treated term $F_0$ to a third implicit direction $F_3$ and setting $F_0\equiv0$. Two properties survive the promotion. The order behavior is favourable: the MCS and HV schemes are second order for any $\theta$ and their order proofs never use $F_0\neq0$, while the Douglas scheme, which is only first order with an explicit mixed term, becomes second order at $\theta=1/2$. In the scalar test with $z_j=\Delta t\,\lambda_j$ and $z=\sum_j z_j$, induction over the stages gives $Y_j=1+z+\theta z(z_1+\dots+z_j)+O(z^3)$, hence a stability function $1+z+\theta z^2+O(z^3)$, which matches $e^z$ to second order precisely when $\theta=1/2$. Exact mass conservation survives as well, since $\bm1^\top F_j=0$ for every block implies $\bm1^\top(\mathcal I-\theta\Delta t\,F_j)^{-1}=\bm1^\top$ and every stage line is mass-neutral.

Positivity does not survive, and the obstruction is structural rather than a matter of how the mixed derivative is treated. A stage of a stabilizing-correction scheme has the form
\begin{equation}\label{eq:sccstage}
Y_j = \bigl(\mathcal I-\theta\Delta t\,F_j\bigr)^{-1}\bigl[\,Y_{j-1}-\theta\Delta t\,F_j\,\bm p^{\,n}\,\bigr],
\end{equation}
and the operand in brackets carries a subtraction. Under the conservative closure $\bm1^\top F_j=0$, so $F_j\bm p^{\,n}$ sums to zero and, unless $\bm p^{\,n}$ is stationary for $F_j$, has entries of both signs. The operand is therefore signed for every nonnegative $\bm p^{\,n}$ of interest, and entrywise nonnegativity of the matrix in front of it constrains its action on the nonnegative cone only. The positivity of the resolvent, or of any other factor placed in that position, is simply not used. Replacing the resolvent by the Pad\'e$(0,2)$ map of \cref{ssec:pade} changes nothing in this respect and costs order, because the cancellation in the order argument above uses $(1-w)^{-1}=1+w+w^2+O(w^3)$ whereas $r_{02}(w)=1+w+w^2/2+O(w^3)$. The resulting diagonal defect is closed-form and can be repaired by adding $\tfrac{\theta^2\Delta t^2}{2}F_j^2\bm p^{\,n}$ to the right-hand side of \eqref{eq:sccstage} at the cost of one banded matrix-vector product, which restores second order and adds one further signed contribution to the operand. The subtractions are not confined to the explicit predictor $Y_0=(\mathcal I+\Delta t\,F)\bm p^{\,n}$; they are the mechanism by which a stabilizing correction achieves its stability, and they therefore appear at every stage of the scheme.

The dividing line is therefore multiplicative against affine. A step written as a product of rational maps of the generators inherits entrywise nonnegativity factor by factor, as in \cref{prop:posadi}, and transposes factor by factor into a consistent scheme for the dual equation, as in \cref{ssec:transpose}. A step written as an affine combination of its stage vectors admits neither argument, and one is left to analyse the assembled transition matrix directly. This is also why the forward construction of \cite{Itkin2015Forward}, which recovers a forward scheme by transposing the explicitly assembled transition matrix of a backward HV or MCS step, is not needed for \eqref{eq:dfadi}. Every factor there is already an explicit rational function of a generator, and the transposition is immediate.

\begin{proposition}[Exact conservation for every step size]\label{prop:massadi}
Assume the conservative closures $\bm1^\top L_x=\bm1^\top L_y=0$ and $\bm1^\top A_{xy}=0$ (edge and corner rows included). Then every factor of \eqref{eq:dfadi} is a rational function $r$ of a zero-column-sum generator with $r(0)=1$, hence $\bm1^\top\Phi_\alpha=\bm1^\top$ and $\bm1^\top\Phi_{xy}=\bm1^\top$, and the composite propagator satisfies $\bm1^\top\mathcal S_{\mathrm{ADI}}(\Delta t)=\bm1^\top$ for every $\Delta t>0$: discrete mass is conserved exactly, irrespective of the positivity windows. For the plain one-sided closure of the cross stencil the statement holds up to the $O(\Delta t\,h)$ boundary leakage quantified in \cite{ItkinDF2026}.
\end{proposition}

\begin{proof}
See \cref{app:massadi}.
\end{proof}

\subsection{Positivity of the composite step} \label{ssec:posadi}

\begin{proposition}[Conditional positivity]\label{prop:posadi}
Suppose the step size $\Delta t$ satisfies simultaneously
\begin{enumerate}
\item[\textup{(i)}] $\tfrac{\Delta t}{2}\ge\gamma_r^{(\alpha)}$ for $\alpha\in\{x,y\}$, the Pad\'e windows of $L_x$, $L_y$ from \cref{thm:pade} (which lift to $A_x$, $A_y$ with the same thresholds by the Kronecker structure);
\item[\textup{(ii)}] $\Delta t\le\Theta$, the window of the central factor from \cite{ItkinDF2026}.
\end{enumerate}
Then every factor of \eqref{eq:dfadi} is entrywise nonnegative, the composite propagator is nonnegative with unit column sums, i.e. column-stochastic, and the scheme is positive and $\ell_1$-nonexpansive on the nonnegative cone. The admissible region $[2\gamma_r,\Theta]$, with $\gamma_r=\max_\alpha\gamma_r^{(\alpha)}$, is bounded below by the directional thresholds and above by the central factor's window, structurally the same as the region $[2\tau_0,\Theta]$ of the exponential scheme in \cite{ItkinDF2026}. For the Crank--Nicolson realization of the directional factors no per-factor guarantee is available (\cref{rem:cnempty}), and its positivity is reported empirically in \cref{sec:numerics}.
\end{proposition}

\begin{proof}
See \cref{app:posadi}.
\end{proof}

\begin{myremark}[The admissible step region]\label{rem:region}
The admissible region $[2\gamma_r,\Theta]$ is the intersection of the one-sided directional windows from below and the central factor's window from above. Whether it is nonempty on a given mesh is a computable question: both thresholds are explicit functions of the assembled matrices and can be evaluated once per problem at negligible cost, and the comparison $2\gamma_r\le\Theta$ replaces the Cayley nonemptiness question, which our experiments resolved negatively (\cref{rem:cnempty}). When the region is empty, three fallbacks are available: (i) the backward Euler directional factors, first order but with the one-sided window of \cref{thm:resolvent}, which by \cref{crossWindow} is the fallback that actually delivers a nonempty region once the mesh is fine enough; (ii) the exponential directional factors of \cite{ItkinDF2026}, which lower the left endpoint to $2\tau_0$ at Krylov cost, though on a strongly coupled problem $2\tau_0$ may itself exceed $\Theta$, as it does throughout \cref{crossWindow}; (iii) accepting bounded negative undershoots monitored at runtime, with mass still conserved exactly by \cref{prop:massadi} under the conservative closure. When the region is empty the criterion still does useful work: it certifies that the empirical positivity of the Pad\'e$(0,2)$ factors is a property of the multiplicative structure (\cref{ssec:relation}), the same structure the stabilizing-correction schemes lack. The numerical section maps the admissible regions for all benchmark problems.
\end{myremark}

\subsection{Complexity} \label{ssec:complexity}

Each directional factor of \eqref{eq:dfadi} costs $N_y$ (resp.\ $N_x$) banded solves of size $N_x$ (resp.\ $N_y$), so the four directional sweeps of a step are $O(N_xN_y)$ with a small constant, and the central factor adds $k_{\mathrm{in}}$ Picard iterations of the same order. The exponential scheme replaces each directional sweep by a Krylov approximation of the action, at $O(m^2N+m^3)$ with $m$ the subspace dimension. \Cref{cost} measures both, together with the two comparators.
\begin{table}[!htb]
\centering
\scalebox{0.85}{
\begin{tabular}{|c|c|c|r|r|r|r|c|}
\toprule
\rowcolor[rgb]{ .792,  .929,  .984}
\multirow{2}{*}{$N$} & \multirow{2}{*}{$m$} & \multirow{2}{*}{$k_{\mathrm{in}}$} & \multicolumn{4}{c|}{\textbf{wall time per step, ms}} & \multirow{2}{*}{\textbf{speedup}} \\
\cmidrule{4-7} \rowcolor[rgb]{ .792,  .929,  .984}
 & & & \multicolumn{1}{c|}{DF-Strang} & \multicolumn{1}{c|}{DF-ADI} & \multicolumn{1}{c|}{HV} & \multicolumn{1}{c|}{MCS} & \\
\hline
 48 &  7 & 5.0 &  11.35 & 1.15 & 0.49 & 0.47 &  9.9 \\
\hline
 96 &  8 & 4.0 &  26.34 & 1.68 & 0.92 & 1.01 & 15.7 \\
\hline
160 & 10 & 4.0 &  65.46 & 3.33 & 2.03 & 2.17 & 19.7 \\
\hline
\rowcolor[rgb]{ .557,  .851,  .451}
240 & 12 & 3.2 & 189.37 & 5.92 & 4.04 & 4.35 & 32.0 \\
\bottomrule
\end{tabular}}
\caption{Cost per step at $\Delta t = 0.05$ in the moderately correlated Regime~I of \cref{ouConv}, $\rho(t)=0.8+0.1\cos(0.7t)$, with the Pad\'e$(0,2)$ directional factors of the flagship. Here $m$ is the Krylov dimension required by the directional exponential at tolerance $10^{-8}$ and $k_{\mathrm{in}}$ is the average number of Picard iterations of the central factor. DF-Strang is timed with an Arnoldi approximation of the action, not with a dense exponential, so the comparison is against the algorithm the companion paper actually uses. The last column is the ratio of the first two timings.}
\label{cost}%
\end{table}%

Two features of \cref{cost} carry the complexity claim. The Krylov dimension grows with the mesh, from $7$ to $12$ over the range tested, so the directional part of the exponential scheme is $O(m^2N)$ with a constant that is itself mesh dependent. The Picard count moves the other way, from $5.0$ down to $3.2$, confirming the mesh-robust bound of \cite{Itkin3D,ItkinDF2026} and leaving the ADI step genuinely linear. The resulting speedup is not a fixed factor but a growing one, from $10$ at $N=48$ to $32$ at $N=240$. Over this range the DF-ADI timings themselves grow more slowly than the unknown count, so they are consistent with the $O(N_xN_y)$ bound without establishing its asymptotics; what they do establish is that the ADI cost is insensitive to the mesh in the way the Krylov cost is not.

One qualification should be attached to any wall-time statement of this kind. In the separable case the directional exponentials are exact, so the exponential scheme is the more accurate propagator per step, and a comparison at equal step size flatters the ADI scheme. \Cref{crossConv} shows the two agreeing to three digits at every level, so at this accuracy the ratio in \cref{cost} is also the ratio at fixed error; in regimes where the two separate, the fixed-error comparison is the one to quote.

\section{Backward Kolmogorov equation and jumps} \label{sec:BKE}

The scheme \eqref{eq:dfadi} advances the probability density of the forward
Fokker--Planck equation. Two extensions are required in practice: the backward
Kolmogorov equation, whose generator is the formal transpose of the forward one,
and jump-diffusion models, which add a nonlocal operator $\mathcal L_J$ to the
splitting. Both extensions preserve the structural properties of the forward
scheme: second-order accuracy, exact conservation, and conditional positivity,
because transposition and jump insertion are compatible with the rational-map
analysis of \cref{sec:resolvent}. The admissible-step conditions, however, are
not inherited unchanged. The jump factor has thresholds of its own, and for the
Pad'e$(0,2)$ map these are substantially more restrictive than the directional
thresholds, as demonstrated in \cref{ssec:numjumps}. For jump operators, the
discretizations developed in \cite{ItkinBook} produce generators $J$ that are
either Metzler (and hence yield an M-matrix resolvent), as for Merton or Kou
jumps, or eventually nonnegative (EN-matrix), as for CGMY with $Y>0$. In the
Metzler case, the exponential and resolvent are nonnegative for every step size,
whereas the Pad'e$(0,2)$ factor retains a strictly positive threshold, as
established by \cref{rem:pademetzler} and quantified in \cref{ssec:numjumps}. In
the EM case, each factor has a one-sided admissible window, exactly as in
\cref{thm:pade}. Following the paradigm of \cite{ItkinBook}, the jump step is
placed at the centre of the Strang composition.

\subsection{Transpose duality} \label{ssec:transpose}

The backward propagator is obtained by transposing each factor of
\eqref{eq:dfadi} and reversing the order:
\begin{equation}\label{eq:dfadiBKE}
\mathcal S_{\mathrm{ADI}}^{\mathrm{bwd}}(\Delta t)
= \Phi_x^\top\,\Phi_y^\top\,\Phi_{xy}^\top\,\Phi_y^\top\,\Phi_x^\top .
\end{equation}
Because $(\mathcal I-\gamma L^\top)^{-1}=((\mathcal I-\gamma L)^{-1})^\top$,
entrywise nonnegativity of every factor, and hence the positivity windows of
\cref{prop:posadi}, transfer to the backward scheme unchanged.  The conserved
left null vector $\bm1^\top$ of the forward generators becomes the right null
vector of the backward ones, which is the discrete statement that the backward
propagator preserves constants.

When jumps are present the transposed jump factor $\Phi_J^\top$ is placed
symmetrically in the backward composition.  Under the constructions of
\cite{ItkinBook} the discretized jump generator satisfies $\bm1^\top J=0$, so
its transpose $J^\top$ has zero row sums and the rational map $r(sJ^\top)$
preserves the constant vector $\bm1$.  Consequently the backward jump step
inherits the unconditional or eventual positivity of the forward jump step
without additional analysis.

A word of caution concerns boundary conditions.  The forward density is
computed with a conservative zero-flux closure, whereas the backward value
function typically carries Dirichlet or other inflow conditions.  The asymmetry
between density and price variables, and the treatment of the jump term at the
boundary, are discussed in detail in \cite{Itkin2015Forward,ItkinBook}.

\subsection{Jumps} \label{ssec:jumps}

With the jump operator inserted at the centre of the directional sweeps the
symmetric composition reads
\begin{equation}\label{eq:dfadi-jump}
\bm p^{\,n+1}
= \Phi_x\!\Bigl(\tfrac{\Delta t}{2}\Bigr)\,
  \Phi_y\!\Bigl(\tfrac{\Delta t}{2}\Bigr)\,
  \Phi_J\!\Bigl(\tfrac{\Delta t}{2}\Bigr)\,
  \Phi_{xy}(\Delta t)\,
  \Phi_J\!\Bigl(\tfrac{\Delta t}{2}\Bigr)\,
  \Phi_y\!\Bigl(\tfrac{\Delta t}{2}\Bigr)\,
  \Phi_x\!\Bigl(\tfrac{\Delta t}{2}\Bigr)\,
  \bm p^{\,n}.
\end{equation}
The factors $\Phi_\alpha$ and $\Phi_{xy}$ are defined as in \cref{sec:dfadi}, while
the jump factor is $\Phi_J(s)=r(sJ)$. For frozen coefficients, the jump flow
$e^{sJ}$ is exact in time, so two half-steps compose to the full step exactly;
for a rational approximation, the symmetric arrangement preserves second-order
accuracy. \Cref{ssec:numjumps} presents the corresponding one-dimensional
experiments for the Merton and Kou kernels. In this setting, the jump generator
is the negative of an M-matrix, so \cref{cor:rinf} applies directly: no additional
positivity condition is required for either the exponential or the resolvent,
whereas the Pad'e map still requires a finite threshold.

\myparagraph{M-matrix and EM-matrix jump generators.}
For compound-Poisson jumps of Merton (Gaussian) or Kou (double-exponential) type,
the matrix $J$ produced by the discretizations in \cite{ItkinBook} is Metzler
with zero column sums; equivalently $-J$ is a singular M-matrix.  Consequently
$e^{sJ}$ and the resolvent $(\mathcal I-sJ)^{-1}$ are entrywise nonnegative for
every $s>0$, so $\tau_0^{(J)}=\gamma_0^{(J)}=0$ and the exponential and
backward-Euler jump factors are unconditionally nonnegative, exactly as the
backward-Euler directional factors of the M-matrix variant.

The Pad'e$(0,2)$ map is \emph{not} covered by this argument. Its denominator
$\mathcal I-sJ+\tfrac{s^2}{2}J^2$ is not an M-matrix because $J^2$ destroys the
required sign pattern outside the original sparsity bands, while
\cref{rem:pademetzler} shows that a Metzler generator alone does not ensure
entrywise positivity of $r_{02}$. The threshold $\gamma_r^{(J)}$ is therefore
strictly positive and, in practice, relatively large: \cref{jumpTab} gives
$\gamma_r^{(J)}=11.4$ for Merton and $48.5$ for Kou, one to two orders of
magnitude larger than the directional thresholds reported in \cref{ouThresh}.
Thus, for jump models, the resolvent is preferable to Pad'e$(0,2)$ for the jump
step; see \cref{ssec:numjumps}.

\myparagraph{Implementation and complexity.}  The jump matrix $J$ is dense,
reflecting the nonlocal nature of the jump operator.  For translation-invariant
kernels the action of $J$ on a vector is a discrete convolution, which can be
evaluated by FFT in $O(N\log N)$ operations.  The Pad\'e$(0,2)$ jump factor then
requires two FFT-based solves per step, or, if the jump step is treated by the
exponential, a single Krylov action on the dense matrix.  In either case the
directional and mixed sweeps remain $O(N)$, so the asymptotic cost per step is
governed by the jump treatment.

\myparagraph{Mass conservation and martingale corrections.}  Under the
conservative closure $\bm1^\top J=0$ the jump factor satisfies
$\bm1^\top\Phi_J(s)=\bm1^\top$ for every rational map with $r(0)=1$, and the
composite step \eqref{eq:dfadi-jump} conserves mass exactly.  For models with
state-dependent jump intensity or measure, the natural discretization may fail
to have exact zero column sums.  Following \cite{ItkinBook}, a diagonal
correction $\mathcal D$ can be added so that $\bm1^\top(J+\mathcal D)=0$; the
corrected jump factor $\tilde\Phi_J(s)=r\bigl(s(J+\mathcal D)\bigr)$ then
restores exact conservation.  In the backward equation the same correction
guarantees that the transposed operator preserves constants, i.e. the discrete
martingale property holds.

\section{Numerical experiments} \label{sec:numerics}

This section reports the numerical experiments. We rerun the benchmarks of \cite{ItkinDF2026} unchanged, so that the comparison with the exponential scheme is direct. Six propagators are compared throughout, joined in the one-dimensional tests of \cref{ssec:kramers} and \cref{ssec:peclet} by a seventh, a centred Crank--Nicolson baseline that does not use the DF operators. DF-Strang is the exponential scheme of \cite{ItkinDF2026}, with the directional exponentials applied by an Arnoldi approximation of the action. DF-ADI is \eqref{eq:dfadi} with the Pad\'e$(0,2)$ directional factors \eqref{eq:dirfactor}, and we also run it with the Crank--Nicolson factors, which are positive in practice but carry no per-factor guarantee (\cref{rem:cnempty}), and with the backward Euler factors, which are the first-order fallback of \cref{rem:region}. The last two are the Hundsdorfer--Verwer and modified Craig--Sneyd schemes with the explicit one-sided cross stencil, run at $\theta = \frac12 + \frac{\sqrt3}{6}$ and $\theta = \frac13$ respectively. The reported metrics are the discrete $L_2$ error against the reference solution, the observed spatial and temporal orders, the most negative value attained over the run, the mass defect and the wall time per step. All computations were performed in Matlab R2025b on an Apple silicon machine. The supporting code is available at
\url{https://github.com/rakhymzhan11/DF-ADI}.

\subsection{One-dimensional benchmark: the Ornstein--Uhlenbeck process} \label{ssec:ou}

In one dimension the mixed block is absent, so this benchmark isolates the directional factors and therefore the $r(\infty)$ criterion itself. The Ornstein--Uhlenbeck process
\begin{equation} \label{eq:ou}
\partial_t p = \kappa\,\partial_x[(x-m)p] + \tfrac12 \sigma^2 \partial_{xx} p, \qquad \mu(x) = -\kappa(x-m), \quad D = \tfrac12\sigma^2,
\end{equation}
has the Gaussian transition density $\mathcal{N}(\bar m(t), v(t))$ with $\bar m(t) = m + (x_0-m)e^{-\kappa t}$ and $v(t) = \frac{\sigma^2}{2\kappa}(1-e^{-2\kappa t})$, which serves as the exact reference. We take $\kappa = \sigma = 1$, $m=0$ on the domain $[-6,6]$ with the conservative closure, and integrate to $T=0.5$ from the initial datum $p_0 = \mathcal{N}(\bar m(t_0), v(t_0))$ at age $t_0 = 0.15$.

\test{1}{Orders of convergence.} \Cref{ouConv} reports the three refinement paths separately, since each isolates a different error. Panel (a) refines the mesh at a fixed small step $\Delta t = 2\cdot10^{-4}$ and measures against the analytic density, so it sees the spatial defect alone. Panel (b) fixes the mesh at $n=401$ and refines the step against the exact-in-time semidiscrete solution $e^{T\cdot L}\bm p_0$ computed on the same mesh, so it sees the time integrator alone. Panel (c) refines both jointly as $\Delta t \sim h$ against the analytic density, which is the regime the assembled scheme is used in.

The temporal panel is the informative one. The exponential column sits at the round-off floor, below $4\cdot10^{-14}$, at every step size, as it must, since on a frozen-coefficient one-dimensional problem $e^{\Delta t L}$ is exact in time. That column therefore validates the reference rather than the scheme. Against it the Pad\'e$(0,2)$ factor converges at rate $1.98$, Crank--Nicolson at $2.00$ and backward Euler at $1.00$, which is the expected ordering and confirms that the second-order accuracy of the flagship is not paid for anywhere. Under joint refinement all three second-order propagators reproduce the analytic density at rate $1.98$ while backward Euler drops to $0.96$, so the first-order fallback of \cref{rem:region} costs a full order, as expected. Mass is conserved to round-off by every propagator, independently of positivity: $3\cdot10^{-14}$ or better for the exponential, backward Euler and Crank--Nicolson factors, and $1\cdot10^{-12}$ for Pad\'e$(0,2)$, whose extra $\tfrac{\Delta t^2}{2}L^2$ term costs a little conditioning. This is \cref{prop:massadi} in the one-dimensional setting.
\begin{table}[!htb]
\begin{center}
\scalebox{0.9}{
\begin{tabular}{|c|c|r|c|r|c|r|c|r|c|}
\hline
\rowcolor[rgb]{ .792,  .929,  .984}
\multirow{2}{*}{$n$} & \multirow{2}{*}{$h$} & \multicolumn{2}{c|}{\textbf{exponential}} & \multicolumn{2}{c|}{\textbf{Pad\'e$(0,2)$}} & \multicolumn{2}{c|}{\textbf{Backward Euler}} & \multicolumn{2}{c|}{\textbf{Crank--Nicolson}} \\
\cmidrule{3-10} \rowcolor[rgb]{ .792,  .929,  .984}
 & & \multicolumn{1}{c|}{$\|e_h\|_2$} & rate & \multicolumn{1}{c|}{$\|e_h\|_2$} & rate & \multicolumn{1}{c|}{$\|e_h\|_2$} & rate & \multicolumn{1}{c|}{$\|e_h\|_2$} & rate \\
\hline
101 & 0.1200 & 4.029e-03 & ---  & 4.029e-03 & ---  & 3.998e-03 & ---  & 4.029e-03 & ---  \\
\hline
201 & 0.0600 & 1.089e-03 & 1.89 & 1.089e-03 & 1.89 & 1.058e-03 & 1.92 & 1.089e-03 & 1.89 \\
\hline
401 & 0.0300 & 2.817e-04 & 1.95 & 2.817e-04 & 1.95 & 2.535e-04 & 2.06 & 2.817e-04 & 1.95 \\
\hline
801 & 0.0150 & 7.153e-05 & 1.98 & 7.152e-05 & 1.98 & 5.634e-05 & 2.17 & 7.153e-05 & 1.98 \\
\hline
\end{tabular}}

\vspace{0.6\baselineskip}
(a) spatial refinement at fixed $\Delta t = 2\cdot 10^{-4}$, error against the analytic density
\vspace{\baselineskip}

\scalebox{0.9}{
\begin{tabular}{|c|c|r|c|r|c|r|c|r|c|}
\hline
\rowcolor[rgb]{ .792,  .929,  .984}
\multirow{2}{*}{$N_t$} & \multirow{2}{*}{$\Delta t$} & \multicolumn{2}{c|}{\textbf{exponential}} & \multicolumn{2}{c|}{\textbf{Pad\'e$(0,2)$}} & \multicolumn{2}{c|}{\textbf{Backward Euler}} & \multicolumn{2}{c|}{\textbf{Crank--Nicolson}} \\
\cmidrule{3-10} \rowcolor[rgb]{ .792,  .929,  .984}
 & & \multicolumn{1}{c|}{$\|e_h\|_2$} & rate & \multicolumn{1}{c|}{$\|e_h\|_2$} & rate & \multicolumn{1}{c|}{$\|e_h\|_2$} & rate & \multicolumn{1}{c|}{$\|e_h\|_2$} & rate \\
\hline
 10 & 0.0500 & 3.734e-14 & ---  & 6.537e-04 & ---  & 1.222e-02 & ---  & 3.867e-04 & ---  \\
\hline
 20 & 0.0250 & 3.751e-14 & ---  & 1.768e-04 & 1.89 & 6.181e-03 & 0.98 & 9.661e-05 & 2.00 \\
\hline
 40 & 0.0125 & 3.727e-14 & ---  & 4.614e-05 & 1.94 & 3.108e-03 & 0.99 & 2.415e-05 & 2.00 \\
\hline
 80 & 0.0063 & 3.633e-14 & ---  & 1.180e-05 & 1.97 & 1.558e-03 & 1.00 & 6.037e-06 & 2.00 \\
\hline
160 & 0.0031 & 3.703e-14 & ---  & 2.983e-06 & 1.98 & 7.803e-04 & 1.00 & 1.509e-06 & 2.00 \\
\hline
\end{tabular}}

\vspace{0.6\baselineskip}
(b) temporal refinement at fixed $n=401$, error against the exact-in-time $e^{T\cdot L}\bm p_0$
\vspace{\baselineskip}

\scalebox{0.9}{
\begin{tabular}{|c|c|c|r|c|r|c|r|c|r|c|}
\hline
\rowcolor[rgb]{ .792,  .929,  .984}
\multirow{2}{*}{$n$} & \multirow{2}{*}{$h$} & \multirow{2}{*}{$N_t$} & \multicolumn{2}{c|}{\textbf{exponential}} & \multicolumn{2}{c|}{\textbf{Pad\'e$(0,2)$}} & \multicolumn{2}{c|}{\textbf{Backward Euler}} & \multicolumn{2}{c|}{\textbf{Crank--Nicolson}} \\
\cmidrule{4-11} \rowcolor[rgb]{ .792,  .929,  .984}
 & & & \multicolumn{1}{c|}{$\|e_h\|_2$} & rate & \multicolumn{1}{c|}{$\|e_h\|_2$} & rate & \multicolumn{1}{c|}{$\|e_h\|_2$} & rate & \multicolumn{1}{c|}{$\|e_h\|_2$} & rate \\
\hline
101 & 0.1200 & 10 & 4.029e-03 & ---  & 3.636e-03 & ---  & 1.021e-02 & ---  & 4.296e-03 & ---  \\
\hline
201 & 0.0600 & 20 & 1.089e-03 & 1.89 & 9.834e-04 & 1.89 & 5.567e-03 & 0.88 & 1.155e-03 & 1.90 \\
\hline
401 & 0.0300 & 40 & 2.817e-04 & 1.95 & 2.542e-04 & 1.95 & 2.940e-03 & 0.92 & 2.979e-04 & 1.96 \\
\hline
801 & 0.0150 & 80 & 7.153e-05 & 1.98 & 6.452e-05 & 1.98 & 1.514e-03 & 0.96 & 7.558e-05 & 1.98 \\
\hline
\end{tabular}}

\vspace{0.6\baselineskip}
(c) joint refinement $\Delta t \sim h$, error against the analytic density
\end{center}
\caption{Convergence on the Ornstein--Uhlenbeck benchmark \eqref{eq:ou} with $\kappa=\sigma=1$, $m=0$, domain $[-6,6]$, $T=0.5$, conservative closure. In panel (b) the exponential is exact in time and its column is the round-off floor, which validates the reference; no rate is quoted for it. Mass is conserved to $3\cdot10^{-14}$ (exponential), $1\cdot10^{-12}$ (Pad\'e), $5\cdot10^{-15}$ (backward Euler) and $7\cdot10^{-15}$ (Crank--Nicolson) at the finest level.}
\label{ouConv}
\end{table}

\test{2}{ Thresholds and the admissible step region.} We next compute the thresholds of \cref{sec:resolvent} directly from the assembled matrices. For a given $L$ each threshold is the endpoint of the set on which the corresponding matrix function is entrywise nonnegative, located by bisection on its smallest entry, with entries counted as nonnegative when they exceed $-10^{-12}$. \Cref{ouThresh} reports $\tau_0$ for the exponential, $\gamma_0$ for the backward Euler resolvent, $\gamma_r$ for the Pad\'e$(0,2)$ map, and the two Cayley endpoints $\gamma_-$ and $\gamma_+$ of \cref{thm:cayley}, in the advection-dominated regime $\kappa=6$, $\sigma=0.5$ on $[-2.5,2.5]$, where the cell P\'eclet number exceeds two and the operator is strongly EM.

Three features of the table matter. First, $\tau_0$, $\gamma_0$ and $\gamma_r$ are all strictly positive and finite, so the exponential, backward Euler and Pad\'e factors each carry a genuine one-sided window, in accordance with \cref{thm:resolvent} and \cref{thm:pade}. Second, $\gamma_r$ is of the same order as $\tau_0$ and, unlike $\gamma_0$, does not shrink under refinement, which is the quantitative statement that Pad\'e$(0,2)$ inherits the positivity mechanism of the exponential rather than that of a low-order resolvent. Third, and this is \cref{rem:cnempty}, the Cayley window is empty on every mesh, because the step at which the off-diagonal entries first turn nonnegative already exceeds the step at which the diagonal has descended below one half. The failure is not marginal. Scanning all $\gamma$ on the finest mesh, the least negative entry the Cayley factor ever attains is $-7.7\cdot10^{-3}$, and at $\gamma = \tau_0$, where the exponential is already nonnegative, the Cayley factor still has an entry equal to $-0.98$.

This conclusion applies to the advection-dominated regime tabulated here.
Repeating the sweep for a weakly eventually positive operator, $\kappa=\sigma=1$
with $\mathrm{Pe}\le1.3$, gives the same result on the three coarser meshes. At
$n=161$, however, the endpoints finally cross and a window opens:
$[5.24\cdot10^{-3},,5.93\cdot10^{-3}]$. Its width is only $7\cdot10^{-4}$,
corresponding to a factor of $1.13$ in the substep and remaining orders of
magnitude below the steps at which the one-sided windows occur. Thus, emptiness
of the Cayley window is not a theorem. Rather, \cref{rem:cnempty} should be read
as the empirical conclusion supported by our experiments: on every mesh and in
every regime tested, the trapezoidal factor has either no positivity window or
one too narrow to be practically useful. \Cref{Fwindows} illustrates the
resulting admissible sets.
\begin{table}[!htb]
\centering
\scalebox{0.9}{
\begin{tabular}{|c|c|c|r|r|r|r|r|l|}
\toprule
\rowcolor[rgb]{ .792,  .929,  .984}
$n$ & $h$ & $\mathrm{Pe}_{\max}$ & \multicolumn{1}{c|}{$\tau_0$} & \multicolumn{1}{c|}{$\gamma_0$} & \multicolumn{1}{c|}{$\gamma_r$} & \multicolumn{1}{c|}{$\gamma_-$} & \multicolumn{1}{c|}{$\gamma_+$} & \multicolumn{1}{c|}{\textbf{Cayley window}} \\
\hline
 51 & 0.1000 & 12.00 & 4.439e-01 & 2.906e-02 & 3.115e-01 & 5.813e-02 & 9.059e-03 & empty, $\gamma_->\gamma_+$ \\
\hline
 81 & 0.0625 &  7.50 & 4.029e-01 & 1.649e-02 & 3.141e-01 & 3.297e-02 & 5.393e-03 & empty, $\gamma_->\gamma_+$ \\
\hline
121 & 0.0417 &  5.00 & 3.407e-01 & 8.740e-03 & 3.164e-01 & 1.752e-02 & 3.433e-03 & empty, $\gamma_->\gamma_+$ \\
\hline
\rowcolor[rgb]{ .557,  .851,  .451}
161 & 0.0312 &  3.75 & 2.860e-01 & 5.237e-03 & 3.175e-01 & 1.053e-02 & 2.475e-03 & empty, $\gamma_->\gamma_+$ \\
\bottomrule
\end{tabular}}
\caption{Positivity thresholds of the directional factors for the Ornstein--Uhlenbeck operator in the advection-dominated regime $\kappa=6$, $\sigma=0.5$ on $[-2.5,2.5]$, conservative closure. Here $\tau_0$ belongs to the exponential, $\gamma_0$ to the backward Euler resolvent, $\gamma_r$ to the Pad\'e$(0,2)$ map, and $[\gamma_-,\gamma_+]$ is the Cayley window of \cref{thm:cayley}. All thresholds are computed on the substep, so the admissible steps of the composite scheme are twice these values. The Cayley window is empty on every mesh, which is the data behind \cref{rem:cnempty}.}
\label{ouThresh}%
\end{table}%
\begin{figure}[!htp]
\begin{center}
\hspace*{-0.3in}
\includegraphics[width=0.8\textwidth]{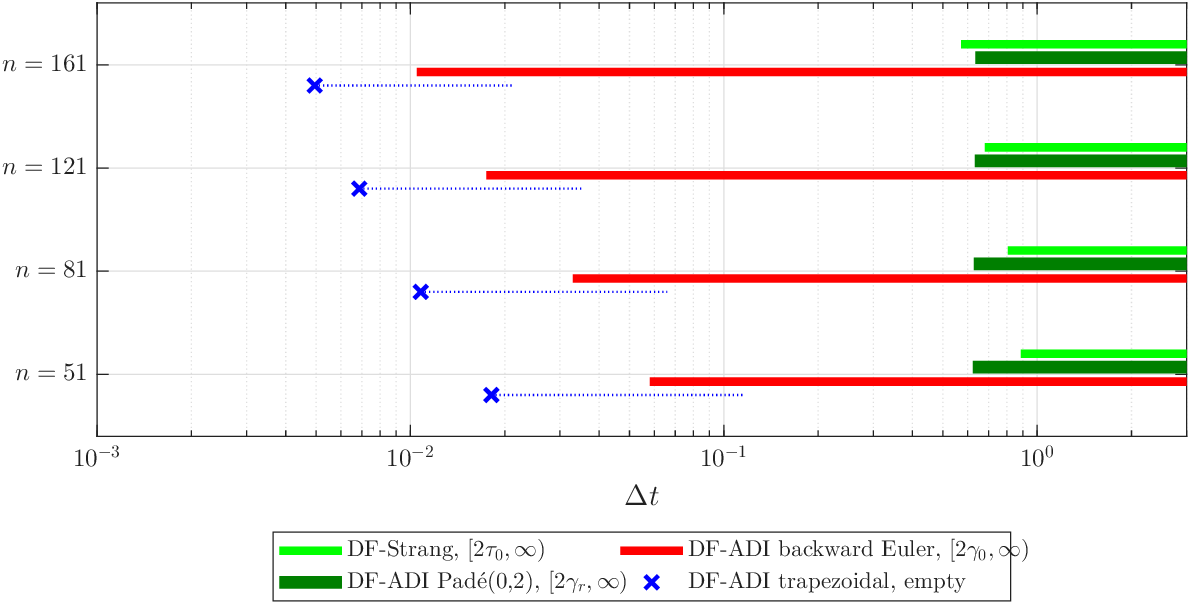}
\end{center}
\caption{Admissible step sets on the Ornstein--Uhlenbeck benchmark of \cref{ouThresh}, one row per mesh. The exponential, Pad\'e$(0,2)$ and backward Euler factors each occupy a one-sided set $[2\tau_0,\infty)$, $[2\gamma_r,\infty)$ and $[2\gamma_0,\infty)$ respectively, since every directional factor of \eqref{eq:dfadi} carries the substep $\Delta t/2$. The trapezoidal factor has no admissible set at all and its collapsed window is marked by a cross.}
\label{Fwindows}
\end{figure}

\begin{myremark}[Numerical certification of the thresholds] \label{rem:certification}
The thresholds above are tolerance thresholds rather than exact ones. For the operators of this benchmark the smallest entries of the stationary projector are of order $e^{-36}$, which is below the resolution of double precision, and the matrix $\mathcal I - \gamma L + \frac{\gamma^2}{2}L^2$ has condition number growing like $\gamma^2\|L\|^2$, so for large $\gamma$ the computed sign of the smallest entry is not reliable. We verified the entries reported here in $40$-digit arithmetic. In particular $\gamma_r$ remains finite and stable under refinement, in agreement with \cref{cor:rinf}(a); a scan carried out in double precision alone would report it as unbounded on the finer meshes, which is an artefact of the conditioning and not a property of the operator.
\end{myremark}

\subsection{Two-dimensional anisotropic Gaussian, strong cross-diffusion} \label{ssec:cross}

We now turn on the mixed derivative and push the correlation towards one, which is where the treatment of the cross term decides the outcome. The drift is linear, $\mu_x = -\vartheta_x(t)x$, $\mu_y = -\vartheta_y(t)y$ with $\vartheta_x(t) = 1.5+0.25\sin t$ and $\vartheta_y(t)=1.5+0.25\cos(0.8t)$, the diffusion is $\Sigma_{xx}=\Sigma_{yy}=1$, and the correlation is $\rho = 0.95$. The initial datum is the Gaussian with mean $(1,-1)^\top$ and covariance $\frac12 \mathcal I$, the domain is $[-6,6]^2$ and $T=0.3$. Since the drift is linear and the diffusion is constant, the exact solution stays Gaussian and its mean and covariance solve a small system of ordinary differential equations, which we integrate to machine precision and use as the reference.

\test{3}{Joint refinement at $\rho=0.95$.} \Cref{crossConv} refines $\Delta t \sim h$ over five levels. The three multiplicative second-order propagators, the exponential and the Pad\'e and Crank--Nicolson realizations of \eqref{eq:dfadi}, are indistinguishable to three digits at every level and converge at rate $2.11$; the backward Euler fallback converges at $2.23$ here because at these steps the error is still dominated by the spatial defect. The two stabilizing-correction comparators behave differently. At this correlation the explicit cross stencil forces a step restriction that the joint path $\Delta t\sim h$ eventually violates, so the Hundsdorfer--Verwer error stops decreasing at the last level and the modified Craig--Sneyd scheme, whose $\theta=\frac13$ damps less, grows without control, reaching $2.37\cdot10^{12}$ at $N=120$. We emphasise that this is a restriction on the step and not a loss of the formal order: on the same mesh $N=120$, halving the step from $N_t=36$ to $N_t=72$ returns the Hundsdorfer--Verwer error to $1.46\cdot10^{-2}$, and it remains there under further halving, so the scheme is simply being run outside its stability region by a path that the multiplicative factors tolerate.

The last row of the table is the positivity comparison. The multiplicative propagators undershoot by about $5\cdot10^{-8}$, six orders of magnitude below the peak density, with the first-order backward Euler factor an order of magnitude worse at $6\cdot10^{-7}$, whereas the explicit-cross comparators undershoot by $4.8\cdot10^{-2}$ and by $2.3\cdot10^{12}$ respectively. The mechanism is the explicit predictor $Y_0 = (\mathcal I + \Delta t F)\bm p^{\,n}$ discussed in \cref{ssec:relation}, and \cref{Fstress} isolates it by sweeping $\rho$ at a fixed mesh.
\begin{table}[!htb]
\centering
\scalebox{0.9}{
\begin{tabular}{|c|c|r|r|r|r|r|r|}
\toprule
\rowcolor[rgb]{ .792,  .929,  .984}
\multirow{2}{*}{$N$} & \multirow{2}{*}{$h$} & \multicolumn{4}{c|}{\textbf{multiplicative factors}} & \multicolumn{2}{c|}{\textbf{explicit cross stencil}} \\
\cmidrule{3-8} \rowcolor[rgb]{ .792,  .929,  .984}
 & & \multicolumn{1}{c|}{exponential} & \multicolumn{1}{c|}{Pad\'e$(0,2)$} & \multicolumn{1}{c|}{Crank--Nicolson} & \multicolumn{1}{c|}{Backward Euler} & \multicolumn{1}{c|}{HV} & \multicolumn{1}{c|}{MCS} \\
\hline
 32 & 0.3871 & 1.222e-01 & 1.222e-01 & 1.222e-01 & 1.172e-01 & 1.664e-01 & 1.746e-01 \\
\hline
 44 & 0.2791 & 8.468e-02 & 8.469e-02 & 8.468e-02 & 8.057e-02 & 1.185e-01 & 1.256e-01 \\
\hline
 64 & 0.1905 & 4.417e-02 & 4.417e-02 & 4.417e-02 & 4.117e-02 & 5.775e-02 & 6.083e-02 \\
\hline
 88 & 0.1379 & 2.281e-02 & 2.281e-02 & 2.281e-02 & 2.063e-02 & 2.796e-02 & 3.259e+00 \\
\hline
120 & 0.1008 & 1.180e-02 & 1.180e-02 & 1.180e-02 & 1.027e-02 & 3.815e-02 & 2.368e+12 \\
\hline
\rowcolor[rgb]{ .557,  .851,  .451}
\multicolumn{2}{|c|}{\textbf{observed rate}} & 2.11 & 2.11 & 2.11 & 2.23 & --- & --- \\
\hline
\multicolumn{2}{|c|}{$\min_n \bm p$, finest} & -4.72e-08 & -5.91e-08 & -4.12e-08 & -6.45e-07 & -4.83e-02 & -2.32e+12 \\
\bottomrule
\end{tabular}}
\caption{Joint refinement $\Delta t \sim h$ on the anisotropic Gaussian with $\rho = 0.95$, $T=0.3$, domain $[-6,6]^2$, error in the discrete $L_2$ norm against the exact Gaussian. No rate is quoted for the two explicit-cross comparators because at this correlation the joint path leaves their stability region at the last level or two; see the discussion in the text. At $N=120$ the MCS error is $2.37\cdot10^{12}$. The computation does not overflow, and every step remains finite; the scheme is simply unstable on this path, and the solution grows monotonically from a peak amplitude of $0.32$ at the first step to nodal values of order $10^{12}$ at the last, the most negative of which is reported in the bottom row. We quote the number rather than the word \emph{diverges} because the failure is a finite, measurable loss of stability rather than a breakdown of the arithmetic.}
\label{crossConv}%
\end{table}%

\test{4}{ The admissible step region.} \Cref{crossWindow} evaluates the two endpoints of \cref{prop:posadi} on the production meshes at $\rho=0.95$. The three lower endpoints, one per directional factor, are computed as in \cref{ouThresh} from the assembled directional operators. The upper endpoint is the resolved-data window of the central factor, whose leading constraint is nonnegativity of the explicit half $(\mathcal I + \frac{\Delta t}{2}A_{xy})\bm p^{\,n}$. It is therefore a resolution condition: it measures how large a step the mixed stencil tolerates before it overshoots the datum it acts on, and it is a property of that datum as much as of the mesh.

\begin{table}[!htb]
\centering
\begin{tabular}{|c|c|r|r|r|r|l|}
\toprule
\rowcolor[rgb]{ .792,  .929,  .984}
$N$ & $h$ & \multicolumn{1}{c|}{$2\tau_0$} & \multicolumn{1}{c|}{$2\gamma_0$} & \multicolumn{1}{c|}{$2\gamma_r$} & \multicolumn{1}{c|}{$\Theta$} & \multicolumn{1}{c|}{\textbf{nonempty region}} \\
\hline
41 & 0.3000 & 3.115 & 2.477e-01 & 2.067 & 2.121e-03 & none \\
\hline
61 & 0.2000 & 2.574 & 1.196e-01 & 2.118 & 2.235e-02 & none \\
\hline
81 & 0.1500 & 2.096 & 6.731e-02 & 2.141 & 6.593e-02 & none \\
\hline
101 & 0.1200 & 1.720 & 4.280e-02 & 2.155 & 5.805e-02 & backward Euler \\
\hline
121 & 0.1000 & 1.426 & 2.951e-02 & 2.161 & 5.447e-02 & backward Euler \\
\hline
161 & 0.0750 & 1.010 & 1.643e-02 & 2.169 & 5.111e-02 & backward Euler \\
\bottomrule
\end{tabular}
\caption{The endpoints of \cref{prop:posadi} at $\rho=0.95$, one lower endpoint per directional factor. The second-order factors keep an $O(1)$ lower endpoint, while that of backward Euler falls like $h$. The upper endpoint $\Theta$ is a resolution condition on the cross stencil: it \emph{rises} steeply under refinement and then saturates near $5\cdot10^{-2}$, so it is not the parabolic $O(h^2)$ quantity the stencil width suggests. The consequence is a clean split. For the exponential and Pad\'e factors the region is empty on every mesh and refinement does not open it, whereas for backward Euler the endpoints cross between $N=81$ and $N=101$ and the region is nonempty on every mesh thereafter.}
\label{crossWindow}%
\end{table}%

For the second-order factors, the region is empty on every mesh, and this is
structural rather than a defect of the Pad'e map. The two endpoints are
controlled by different mechanisms. The lower endpoint $2\gamma_r$ is an $O(1)$
property of the stationary density of the directional operators and is
essentially mesh-independent, varying only from $2.07$ to $2.17$ across the six
meshes; likewise, $2\tau_0$ decreases only from $3.1$ to $1.0$. In contrast, the
upper endpoint increases from $2.1\cdot10^{-3}$ to $6.6\cdot10^{-2}$ as the mixed
stencil resolves the datum, then levels off near $5\cdot10^{-2}$. It is therefore
a resolution condition that converges to a ratio determined by the shape of the
datum rather than by the mesh. Refinement consequently narrows the gap, but the
endpoints remain separated by a factor of about forty. For these factors,
positivity in this regime is the empirical conclusion recorded in
\cref{crossConv}, together with the fallback conditions of \cref{rem:region}.

Backward Euler is the exception, and it is precisely what gives
\cref{prop:posadi} its content in two dimensions. Its threshold is not an $O(1)$
spectral quantity but a discretization-scale quantity that decreases like $h$.
Thus, $2\gamma_0$ falls below the saturating $\Theta$ between $N=81$ and $N=101$,
and the admissible region $[2\gamma_0,\Theta]$ becomes nonempty on every finer
mesh, widening under further refinement. The composite guarantee is therefore
genuinely available, but only for the first-order variant. This is the Bolley–Crouzeix barrier, \cite{BolleyCrouzeix1978} in its sharpest form for this problem, the parabolic analogue of Godunov's theorem: a composite step can be provably nonnegative or second-order accurate, but not both.

\begin{figure}[!htb]
\begin{center}
\hspace*{-0.3in}
\includegraphics[width=0.75\textwidth]{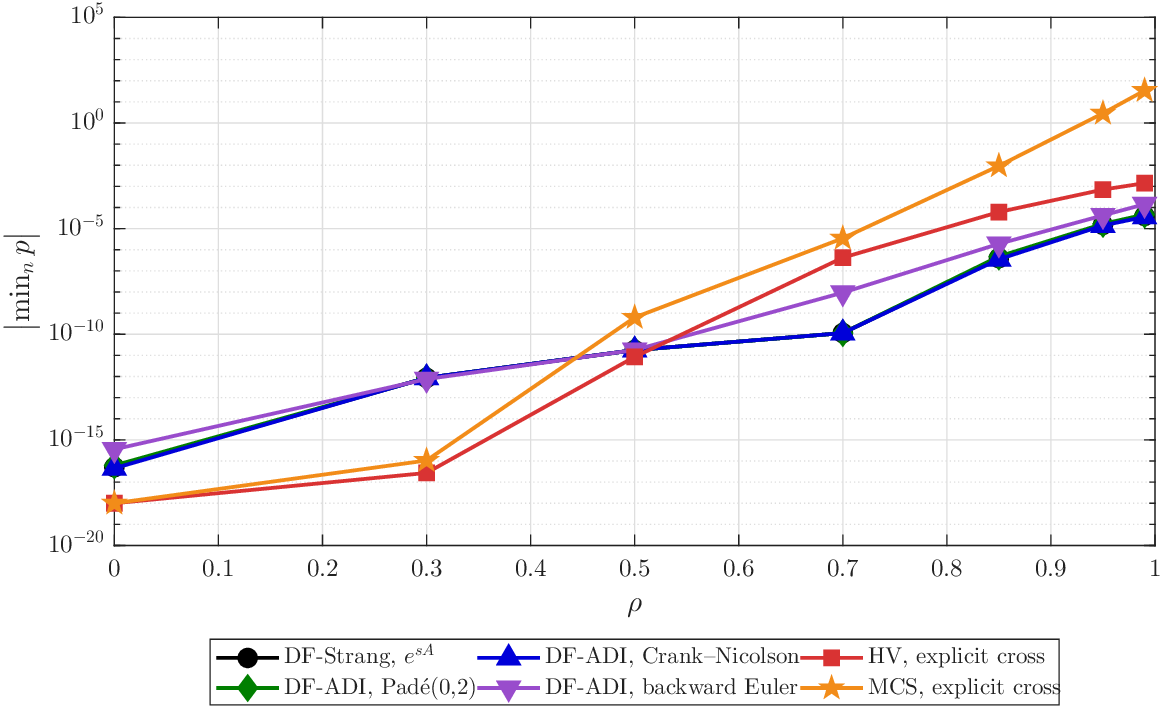}
\end{center}
\caption{Sensitivity of positivity to the strength of the cross-diffusion. Plotted is the most negative nodal value attained at any step of the run, $|\min_n \bm p|$, against the correlation $\rho$ entering $\Sigma_{xy}=\rho\sqrt{\Sigma_{xx}\Sigma_{yy}}$, with everything else held fixed: the mesh $N=88$ in each direction, the step $\Delta t \sim h$ of \cref{crossConv}, and $T=0.3$. Values are floored at $10^{-18}$ for plotting. The three second-order multiplicative propagators remain within $5\cdot10^{-5}$ of nonnegativity up to $\rho = 0.99$, where the first-order backward Euler factor reaches $1.4\cdot10^{-4}$, while the two schemes carrying an explicit cross stencil separate from them beyond $\rho \approx 0.7$ and reach $1.4\cdot10^{-3}$ and $3.4\cdot10^{1}$ respectively. The separation is the explicit predictor of \cref{ssec:relation}, isolated here by varying only the cross term.}
\label{Fstress}
\end{figure}

\subsection{Classical ADI with an implicitly treated mixed derivative} \label{ssec:numrelation}

\Cref{ssec:relation} claims that promoting the mixed block to a third implicit direction preserves the order and the conservation of a stabilizing-correction scheme but not its positivity, the obstruction being the affine stage structure rather than the treatment of the cross term. This subsection measures both halves of that claim. We set $F_0\equiv0$, take $F_1=A_x$, $F_2=A_y$, $F_3=A_{xy}$, and apply $(\mathcal I - \theta\Delta t A_{xy})^{-1}$ by the factorized Picard iteration.

\test{5}{ Order.} \Cref{implOrder} refines the step on a frozen-coefficient problem with $\rho=0.8$, measuring against $e^{T\cdot F}\bm p_0$, so only the time integrator is under test. Two remarks on the setup are needed. First, the mixed factor is applied here by a direct solve rather than by the Picard iteration. The factorized coupling carries a $\Delta t$-independent orientation defect. That defect accumulates over $T/\Delta t$ steps, so at fixed mesh it grows as the step is refined. With the iteration in place every scheme in the table, the multiplicative one included, reports an apparent order near $-0.3$, which measures the defect and not the integrator. Second, the reference is the exponential of the assembled generator, not the analytic density, so the spatial defect cancels identically.

The order claim is confirmed. Douglas, which is only first order when the mixed term is explicit, becomes second order at $\theta=\frac12$ once $F_0$ vanishes, exactly as predicted, and the modified Craig--Sneyd scheme stays second order. The Hundsdorfer--Verwer scheme returns rate $3.05$. This is genuine superconvergence and not a numerical accident: for a step composed of commuting blocks with $F_0\equiv0$ the coefficient of $z^3$ in the stability function vanishes identically at $\theta = \frac12+\frac{\sqrt3}{6}$, which is a root of $\theta^2-\theta+\frac16=0$, and we have verified symbolically that it is nonzero at $\theta=\frac13$ and $\theta=\frac12$. The generic order of the scheme remains two.

The last two rows quantify the remark at the end of \cref{ssec:relation}. Replacing the resolvent in the stages by the Pad\'e$(0,2)$ map costs a full order, since the stage cancellation $(\mathcal I-w)^{-1}(\mathcal I - w) = \mathcal I$ that a stabilizing correction relies on is exact for the resolvent but leaves a defect $-\frac{w^2}{2}$ for $r_{02}$. Adding $\frac{\theta^2\Delta t^2}{2}F_j^2\bm p^{\,n}$ to the stage right-hand side removes that defect and restores the rate. The repair must be applied only to those stages that actually use the Pad\'e map; the mixed stage uses the resolvent and carries no defect, and inserting a repair term there contributes a spurious $O(\Delta t^2)$ per step which pins the observed rate at one.
\begin{table}[!htb]
\centering
\scalebox{0.9}{
\begin{tabular}{|l|r|r|r|r|r|c|}
\toprule
\rowcolor[rgb]{ .792,  .929,  .984}
\multicolumn{1}{|c|}{\textbf{scheme}} & \multicolumn{1}{c|}{$N_t=8$} & \multicolumn{1}{c|}{$16$} & \multicolumn{1}{c|}{$32$} & \multicolumn{1}{c|}{$64$} & \multicolumn{1}{c|}{$128$} & \textbf{rate} \\
\hline
Douglas, $F_0$ explicit, $\theta=\frac12$      & 1.37e-02 & 6.90e-03 & 3.45e-03 & 1.72e-03 & 8.62e-04 & 1.00 \\ \hline
\rowcolor[rgb]{ .557,  .851,  .451}
Douglas, $F_0\equiv0$, $\theta=\frac12$        & 1.33e-03 & 3.32e-04 & 8.26e-05 & 2.06e-05 & 5.14e-06 & 2.00 \\ \hline
HV, $F_0$ explicit                              & 2.37e-03 & 6.53e-04 & 1.71e-04 & 4.39e-05 & 1.11e-05 & 1.98 \\ \hline
HV, $F_0\equiv0$                                & 5.51e-04 & 6.04e-05 & 6.80e-06 & 7.99e-07 & 9.65e-08 & 3.05 \\ \hline
MCS, $F_0\equiv0$                               & 6.08e-04 & 1.51e-04 & 3.75e-05 & 9.33e-06 & 2.33e-06 & 2.00 \\ \hline
HV, $F_0\equiv0$, Pad\'e stages                 & 5.87e-03 & 2.92e-03 & 1.46e-03 & 7.33e-04 & 3.67e-04 & 1.00 \\ \hline
HV, $F_0\equiv0$, Pad\'e stages, repaired       & 5.76e-04 & 6.63e-05 & 7.87e-06 & 9.57e-07 & 1.18e-07 & 3.02 \\ \hline
\rowcolor[rgb]{ .557,  .851,  .451}
DF-ADI \eqref{eq:dfadi}, Pad\'e$(0,2)$            & 5.38e-04 & 1.39e-04 & 3.51e-05 & 8.81e-06 & 2.20e-06 & 2.00 \\
\bottomrule
\end{tabular}}
\caption{Temporal convergence with the mixed block promoted to a third implicit direction, $N=28$, $T=0.2$, frozen coefficients with $\rho=0.8$, error against $e^{T\cdot F}\bm p_0$. Highlighted are the two rows that carry the message of \cref{ssec:relation}: Douglas gains an order when $F_0$ vanishes, and the multiplicative scheme \eqref{eq:dfadi} attains second order without any stage repair.}
\label{implOrder}%
\end{table}%

\test{6}{ Positivity.} \Cref{implPos} reports, alongside the final most negative value, two quantities internal to the stabilizing-correction step: the smallest entry of the explicit predictor $Y_0$, and the smallest entry attained by any stage operand $Y_{j-1}-\theta\Delta t F_j \bm p^{\,n}$ over the run. The second is the quantity \cref{ssec:relation} is about. Making the cross term implicit removes it from the predictor sum, and it does improve the undershoot, by a factor of about $2.6$ in the strongly under-resolved regime III. It does not, however, change the sign structure. The operands remain signed at every stage, so the nonnegativity of the matrix standing in front of them is never used, and no per-factor argument becomes available. The multiplicative scheme \eqref{eq:dfadi} has no analogue of these quantities, since each of its factors acts on the nonnegative cone directly.
\begin{table}[!htb]
\centering
\scalebox{0.9}{
\begin{tabular}{|l|r|r|r|r|}
\toprule
\rowcolor[rgb]{ .792,  .929,  .984}
\multicolumn{1}{|c|}{\textbf{scheme}} & \multicolumn{1}{c|}{$\min_n \bm p$} & \multicolumn{1}{c|}{\textbf{predictor} $\min Y_0$} & \multicolumn{1}{c|}{\textbf{worst operand}} & \multicolumn{1}{c|}{\textbf{mass defect}} \\
\hline
\multicolumn{5}{|l|}{\textit{Regime II, weak cross, resolved} ($N=96$, $N_t=48$)} \\
\hline
\rowcolor[rgb]{ .557,  .851,  .451}
DF-ADI \eqref{eq:dfadi}, Pad\'e$(0,2)$ & -2.67e-02 & \multicolumn{1}{c|}{---} & \multicolumn{1}{c|}{---} & 5.8e-05 \\
\hline
HV, $F_0$ explicit                    & -2.19e-02 & -2.179e-02 & -2.178e-02 & 8.0e-05 \\
\hline
HV, $F_0\equiv0$                      & -2.77e-02 & -2.738e-02 & -2.746e-02 & 5.4e-05 \\
\hline
MCS, $F_0\equiv0$                     & -2.59e-02 & -2.561e-02 & -2.568e-02 & 6.1e-05 \\
\hline
\multicolumn{5}{|l|}{\textit{Regime III, strongly under-resolved} ($N=72$, $N_t=60$)} \\
\hline
\rowcolor[rgb]{ .557,  .851,  .451}
DF-ADI \eqref{eq:dfadi}, Pad\'e$(0,2)$ & -1.28e+00 & \multicolumn{1}{c|}{---} & \multicolumn{1}{c|}{---} & 2.2e-02 \\
\hline
HV, $F_0$ explicit                    & -3.21e+00 & -3.212e+00 & -3.254e+00 & 1.5e-04 \\
\hline
HV, $F_0\equiv0$                      & -1.16e+00 & -1.231e+00 & -1.237e+00 & 2.3e-02 \\
\hline
MCS, $F_0\equiv0$                     & -1.38e+00 & -1.467e+00 & -1.462e+00 & 1.9e-02 \\
\bottomrule
\end{tabular}}
\caption{Positivity diagnostics with the mixed block treated implicitly. The worst operand is the smallest entry attained over the run by any stage operand $Y_{j-1}-\theta\Delta t F_j\bm p^{\,n}$. It stays signed whether or not the cross term is implicit, which is the numerical content of the multiplicative-against-affine distinction of \cref{ssec:relation}. The multiplicative scheme has no such quantity. In this comparison the cross stencil uses the plain one-sided closure shared with the stabilizing-correction comparators. The mass defect of the multiplicative scheme is therefore the $O(\Delta t\,h)$ boundary leakage of \cref{prop:massadi}, not round-off; under the conservative closure it drops to round-off, as in \cref{ouConv} and \cref{kramRate}.}
\label{implPos}%
\end{table}%

\test{7}{ Temporal order of the production step.} The order measurement
of \cref{implOrder} advances the mixed factor by a direct solve, which isolates
the temporal behaviour of the composition but is not how \eqref{eq:dfadi} is run.
\Cref{tempOrder} repeats it in the production configuration, with $\Phi_{xy}$
advanced by the factorized Picard iteration. Coefficients are frozen, so the
assembled generator $F=A_x+A_y+A_{xy}$ is constant and the reference is
$e^{T\cdot F}\bm p^{\,0}$ evaluated by dense exponentiation. The one-sided cross
stencil enters both the scheme and the reference, so the spatial orientation
error cancels identically and what remains is purely temporal.

The two columns behave differently. With the direct mixed solve the scheme is
second order, at $2.00$ over five refinements, which is the temporal order of
\eqref{eq:dfadi} itself. With the Picard solve the error does not merely
stagnate, it grows as $\Delta t$ falls. The factorized iteration does not
converge to $(\mathcal I-sA_{xy})^{-1}\bm b$. It converges, in three to nine
iterations and independently of the tolerance, to the solution of a nearby system
differing from it by $O(h^3)$, a quantity carrying no factor of $\Delta t$.
Accumulated over the $T/\Delta t$ steps of a run, that defect contributes
$O(h^3/\Delta t)$, which must eventually dominate the $O(\Delta t^2)$ truncation
error at fixed mesh. Raising the iteration count does not help, because the
iteration has already converged.

This is not a defect to be repaired but the same trade the rest of the paper is about, appearing inside the mixed solver. The discrepancy originates in the
orientation of the one-sided differences assembling the iteration's right-hand
side, and symmetrizing them does restore second order in the production
configuration. It also degrades positivity by between two and four orders of
magnitude, $\min_n\bm p$ falling from $-5.9\cdot10^{-8}$ to $-2.0\cdot10^{-5}$ at
$\rho=0.95$ on the finest mesh and from $-2.8\cdot10^{-11}$ to $-3.8\cdot10^{-7}$
at $\rho=0.8$, at identical iteration counts and identical mass. The orientation
is doing the same work inside the mixed factor that the one-sided stencil does
outside it, buying nonnegativity at the cost of accuracy.

The practical statement is the one \cref{prop:order} already makes. The scheme is second order under the joint refinement $\Delta t\sim h$, where the $O(h^3)$
defect sits an order below the $O(h^2)$ spatial error and never surfaces, which
is what \cref{crossConv} measures at $2.11$. It is not second order under $\Delta
t\to0$ at fixed mesh, and \eqref{eq:dfadi} should be refined in both arguments
together.
\begin{table}[!htb]
\centering
\scalebox{0.9}{
\begin{tabular}{|c|c|r|c|r|c|}
\toprule
\rowcolor[rgb]{ .792,  .929,  .984}
\multirow{2}{*}{$N_t$} & \multirow{2}{*}{$\Delta t$} &
\multicolumn{2}{c|}{\textbf{Picard mixed solve}} & \multicolumn{2}{c|}{\textbf{direct mixed solve}} \\
\cmidrule{3-6} \rowcolor[rgb]{ .792,  .929,  .984}
 & & \multicolumn{1}{c|}{error} & order & \multicolumn{1}{c|}{error} & order \\
\hline
   4 & 2.500e-02 & 1.343e-03 & ---     & 5.382e-05 & ---  \\
\hline
   8 & 1.250e-02 & 1.791e-03 & $-0.42$ & 1.359e-05 & 1.99 \\
\hline
  16 & 6.250e-03 & 3.123e-03 & $-0.80$ & 3.417e-06 & 1.99 \\
\hline
  32 & 3.125e-03 & 5.762e-03 & $-0.88$ & 8.569e-07 & 2.00 \\
\hline
  64 & 1.563e-03 & 1.037e-02 & $-0.85$ & 2.146e-07 & 2.00 \\
\hline
\rowcolor[rgb]{ .557,  .851,  .451}
 128 & 7.813e-04 & 1.525e-02 & $-0.56$ & 5.368e-08 & 2.00 \\
\bottomrule
\end{tabular}}
\caption{Temporal order of \eqref{eq:dfadi} at fixed mesh $N=64$ in each direction, frozen
coefficients with $\rho=0.8$, $T=0.1$, Pad\'e$(0,2)$ directional factors, error in the discrete
$L_2$ norm against $e^{T\cdot F}\bm p^{\,0}$. The direct mixed solve gives the temporal order of the
composition. The Picard mixed solve carries in addition an $O(h^3)$ per-step defect which does
not scale with $\Delta t$, so its accumulated contribution grows as the step is refined at fixed
mesh. On the coarser mesh $N=48$ the same columns read $3.03\cdot10^{-3}$ to $3.28\cdot10^{-2}$
and $7.70\cdot10^{-5}$ to $7.65\cdot10^{-8}$ respectively, the Picard column scaling with $h$ and
the direct column not.}
\label{tempOrder}%
\end{table}%

\subsection{Kramers escape in a double-well potential} \label{ssec:kramers}

The previous benchmarks run to a fixed short horizon. This one runs long, so that the metastable physics of the equation, rather than the local truncation error, decides whether a scheme is usable. We take the double-well potential $V(x) = \varkappa(x^2-k^2)^2$ with $k=1$, so that $\mu = -V'$, and a small diffusion $D$, which places the problem in the activated-escape regime. The slowest nontrivial eigenvalue of the generator then approaches twice the Kramers rate,
\begin{equation} \label{eq:kramers}
r_K = \frac{\sqrt{V''(x_{\min})\,|V''(x_{\max})|}}{2\pi}\, e^{-\Delta V/D},
\end{equation}
the factor two arising because the symmetric double well relaxes through two equivalent channels.

\test{8}{ Long-horizon positivity.}
\Cref{kramRate} first checks the discretization itself against \eqref{eq:kramers}
as the barrier is raised. The ratio $\lambda_{\mathrm{slow}}/2r_K$ stays within
$10\%$ of unity and approaches one as the barrier grows, after an initial dip at
the lowest barrier, consistent with the finite-barrier corrections omitted by
\eqref{eq:kramers}.

The lower block then integrates $600$ steps to $T=30$ at $\varkappa=1.5$, with a
barrier of $3$, and recovers the escape rate from the decay of the well
population. All five propagators reproduce the reference rate to better than
$0.7\%$. Thus, on this problem the schemes are essentially equivalent in
accuracy, and positivity is the relevant discriminator. The exponential and
backward Euler factors remain nonnegative to round-off throughout the run,
whereas the two trapezoidal factors develop undershoots of order $10^{-2}$ from
the rough initial datum, which persist throughout the metastable phase. The
Pad\'e$(0,2)$ factor lies between them, with a minimum of $-6.8\cdot10^{-5}$. The
step $\Delta t=0.05$ is below the threshold $\gamma_r$ for this operator, so
\cref{thm:pade} provides no guarantee at this step. What remains is the far-band
leak described in \cref{rem:pademetzler}, two orders of magnitude smaller than
the trapezoidal undershoot. The one-sided window therefore behaves as expected,
viewed from below its threshold. \Cref{Fkramers} shows both the decay and the
running minimum over the full run.

Mass deserves a brief qualification. Every propagator in the table conserves mass
to $1\cdot10^{-10}$ or better, with the largest defect occurring for the Pad'e
factor, whose additional $\tfrac{\Delta t^2}{2}L^2$ term introduces some
conditioning sensitivity. This is not a discriminator here: all five schemes use
the conservative closure, for which \cref{prop:massadi} guarantees exact mass
conservation independently of the propagator. The difference becomes visible only
when compared with a closure that lacks this property, such as the plain
one-sided closure of the cross stencil, which leaks at $O(\Delta t,h)$ per step.
\begin{table}[!htb]
\centering
\scalebox{0.9}{
\begin{tabular}{|c|c|r|r|c|}
\toprule
\rowcolor[rgb]{ .792,  .929,  .984}
$\varkappa$ & \textbf{barrier} $\Delta V$ & \multicolumn{1}{c|}{$2 r_K$ \textbf{(Kramers)}} & \multicolumn{1}{c|}{$\lambda_{\mathrm{slow}}$ \textbf{(discrete)}} & \textbf{ratio} \\
\hline
1.0 & 2.00 & 2.4369e-01 & 2.2741e-01 & 0.933 \\
\hline
1.5 & 3.00 & 1.3447e-01 & 1.2121e-01 & 0.901 \\
\hline
2.0 & 4.00 & 6.5959e-02 & 5.9905e-02 & 0.908 \\
\hline
2.5 & 5.00 & 3.0331e-02 & 2.8083e-02 & 0.926 \\
\hline
\rowcolor[rgb]{ .557,  .851,  .451}
3.0 & 6.00 & 1.3390e-02 & 1.2674e-02 & 0.947 \\
\bottomrule
\end{tabular}}

\vspace{\baselineskip}

\scalebox{0.9}{
\begin{tabular}{|l|r|c|r|r|}
\toprule
\rowcolor[rgb]{ .792,  .929,  .984}
\multicolumn{1}{|c|}{\textbf{scheme}} & \multicolumn{1}{c|}{\textbf{recovered rate}} & \textbf{rate}$/\lambda_{\mathrm{ref}}$ & \multicolumn{1}{c|}{$\min_n \bm p$} & \multicolumn{1}{c|}{\textbf{mass drift}} \\
\hline
\rowcolor[rgb]{ .557,  .851,  .451}
DF-Strang, exponential  & 1.2282e-01 & 1.006 & 2.82e-64 & 1.14e-12 \\
\hline
\rowcolor[rgb]{ .557,  .851,  .451}
DF-ADI, Pad\'e$(0,2)$   & 1.2282e-01 & 1.006 & -6.76e-05 & 9.25e-11 \\
\hline
DF-ADI, Crank--Nicolson & 1.2283e-01 & 1.006 & -6.50e-03 & 6.86e-13 \\
\hline
DF-ADI, Backward Euler  & 1.2245e-01 & 1.003 & 2.82e-64 & 4.26e-13 \\
\hline
centred Crank--Nicolson & 1.2145e-01 & 0.995 & -9.22e-03 & 1.90e-12 \\
\bottomrule
\end{tabular}}
\caption{Kramers escape in the double well $V=\varkappa(x^2-k^2)^2$, $k=1$. Upper block: the slowest nontrivial eigenvalue of the discrete generator against twice the Kramers rate \eqref{eq:kramers} as the barrier is raised. Lower block: a long run at $\varkappa=1.5$, $T=30$, $\Delta t=0.05$, $600$ steps, with $\lambda_{\mathrm{ref}}=1.2204\cdot10^{-1}$ the eigenvalue of the same discrete generator. All schemes recover the rate to better than $0.7\%$, so the discriminator is the sign of the solution. The row labelled centred Crank--Nicolson is the standard second-order centred-difference spatial discretization advanced by the trapezoidal rule. It is a naive baseline that does not use the DF operators, and is distinct from the DF-ADI, Crank--Nicolson row above, which does.}
\label{kramRate}%
\end{table}%
\begin{figure}[!htb]
\begin{center}
\includegraphics[width=0.98\textwidth]{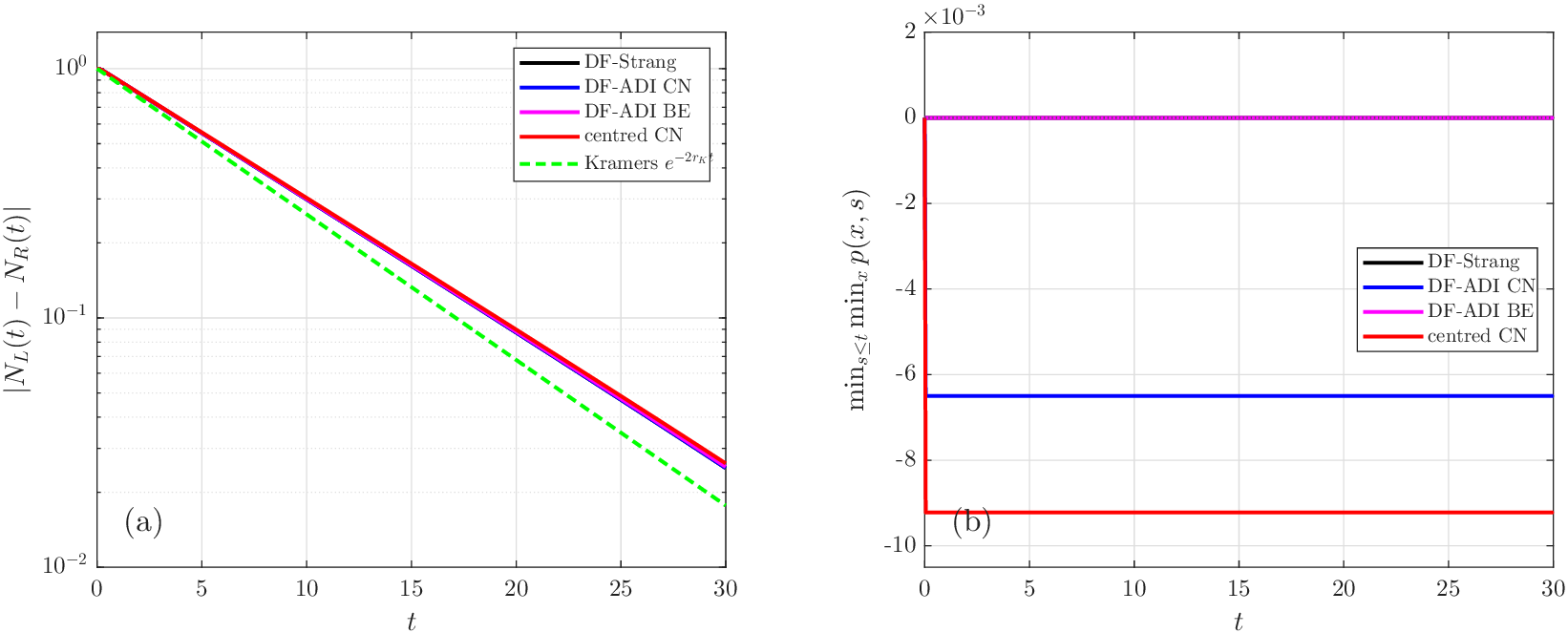}
\end{center}
\caption{Long-time behavior in the double-well potential of \cref{kramRate}, with
$\varkappa=1.5$, $T=30$, and $\Delta t=0.05$. Panel (a): decay of the
inter-well population imbalance $|N_L(t)-N_R(t)|$, where $N_L$ and $N_R$ are
the masses in the left and right wells, compared with the Kramers prediction
$e^{-2r_Kt}$ from \eqref{eq:kramers} (dashed). All five propagators follow the
predicted decay. Panel (b): depth of the undershoot, $\max\bigl(0,-\min_{s\le t}\min_x p(x,s)\bigr)$, shown on a logarithmic axis with a plotting floor at
$10^{-9}$; the four orders of magnitude separating the schemes would be invisible
on a linear scale. The exponential and backward Euler factors remain nonnegative
and therefore stay at the floor, the Pad'e$(0,2)$ factor reaches
$6.8\cdot10^{-5}$, while the two trapezoidal factors reach about $10^{-2}$ during
the initial transient and do not recover. Mass is not shown because all schemes
use the conservative closure, which ensures exact conservation by
\cref{prop:massadi}; the corresponding panel would therefore provide no
discrimination.}
\label{Fkramers}
\end{figure}

\subsection{Advection-dominated regime} \label{ssec:peclet}

The next benchmark increases the cell P'eclet number $\mathrm{Pe}*i=|\mu_i|h/D$
and examines how the directional factors behave as the mesh ceases to resolve the
drift. We fix the Ornstein--Uhlenbeck problem \eqref{eq:ou} with $\sigma=0.5$ on
$n=161$ nodes and vary $\kappa$, so that $\mathrm{Pe}*{\max}$ ranges from $1.2$
to $10$, while integrating a resolved bump to $T=0.25$. Unlike the other
benchmarks, this sweep advances the interior block of the operator with an
absorbing (Dirichlet) closure; the bump remains well away from the boundary, so
the experiment isolates the behavior of the directional factor in the
drift-dominated interior rather than boundary effects. The thresholds reported in
\cref{ouThresh} are those of the conservative operator and are used here only to
set the scale of $\Delta t$.
\begin{table}[!htb]
\centering
\scalebox{0.9}{
\begin{tabular}{|c|c|r|r|r|r|r|}
\toprule
\rowcolor[rgb]{ .792,  .929,  .984}
\multirow{2}{*}{$\kappa$} & \multirow{2}{*}{$\mathrm{Pe}_{\max}$} &
\multicolumn{3}{c|}{\textbf{one-sided window,} $r(\infty)=0$} &
\multicolumn{2}{c|}{\textbf{trapezoidal,} $r(\infty)=-1$} \\
\cmidrule{3-7}
\rowcolor[rgb]{ .792,  .929,  .984}
 & & exponential & Pad\'e$(0,2)$ & Backward Euler & DF-ADI, Crank--Nicolson & centred Crank--Nicolson \\
\hline
 2 &  1.2 & -7.35e-05 & -6.46e-03 & 0.00e+00 &  0.00e+00 &  0.00e+00 \\
\hline
 4 &  2.5 & -3.88e-03 & -3.82e-02 & 0.00e+00 & -1.90e-06 & -6.03e-41 \\
\hline
 6 &  3.8 & -1.89e-02 & -8.88e-02 & 0.00e+00 & -6.81e-06 & -2.66e-05 \\
\hline
 9 &  5.6 & -5.12e-02 & -1.57e-01 & 0.00e+00 & -9.15e-04 & -2.22e-01 \\
\hline
12 &  7.5 & -1.06e-01 & -2.05e-01 & 0.00e+00 & -4.30e-01 & -7.85e-01 \\
\hline
\rowcolor[rgb]{ .557,  .851,  .451}
16 & 10.0 & -1.75e-01 & -2.61e-01 & 0.00e+00 & -1.27e+00 & -1.69e+00 \\
\bottomrule
\end{tabular}}
\caption{P\'eclet sweep on \eqref{eq:ou} with $\sigma=0.5$, $n=161$, $T=0.25$, $\Delta t = 0.01$, reporting $\min_n \bm p$, the most negative nodal value
attained at any step of the run. The step is below the thresholds $2\tau_0$ and
$2\gamma_r$ of \cref{ouThresh}, so the exponential and Pad\'e columns show the
sub-threshold transient of \cref{rem:zenomirror} rather than ringing. Entries of
magnitude $10^{-41}$ in the trapezoidal columns are round-off, not genuine
undershoots.}
\label{peSweep}%
\end{table}

\begin{table}[!htb]
\centering
\scalebox{0.9}{
\begin{tabular}{|l|c|r|r|c|}
\toprule
\rowcolor[rgb]{ .792,  .929,  .984}
\textbf{scheme} &
\multicolumn{1}{c|}{\textbf{stiff}} &
\multicolumn{1}{c|}{$\min_{n}\bm p$} &
\multicolumn{1}{c|}{$\min_n \bm p$} &
\multicolumn{1}{c|}{\textbf{negative}} \\
\rowcolor[rgb]{ .792,  .929,  .984}
 &
\multicolumn{1}{c|}{\textbf{damping} $r(\infty)$} &
\multicolumn{1}{c|}{\textbf{over the run}} &
\multicolumn{1}{c|}{\textbf{at} $t=T$} &
\multicolumn{1}{c|}{\textbf{nodes at} $t=T$} \\
\hline
DF-ADI, Pad\'e$(0,2)$    & 0  & -2.34e-03 & -5.75e-06 & 30 \\
\hline
DF-ADI, Backward Euler   & 0  &  0.00e+00 &  0.00e+00 & 0 \\
\hline
DF-Strang, exponential   & 0  & -2.59e-15 &  0.00e+00 & 0 \\
\hline
\rowcolor[rgb]{ .557,  .851,  .451}
DF-ADI, Crank--Nicolson  & -1 & -6.13e+00 &  0.00e+00 & 0 \\
\bottomrule
\end{tabular}}
\caption{Stiff ringing test: pure diffusion, so the operator is Metzler and no
eventual-positivity threshold is in play, started from a grid-scale spike and
advanced with $\frac{\Delta t}{2}\max_i|L_{ii}| = 2.6$. The last two columns are
instantaneous quantities at the final time and must be read together with the
third. Crank--Nicolson attains by far the worst undershoot \emph{during} the run,
$-6.13$, but the reflected spike has diffused away by $t=T$, so it leaves no
negative nodes behind. Pad\'e$(0,2)$ never rings, yet retains $30$ weakly
negative nodes at $t=T$ of size $10^{-6}$: this is the persistent far-band leak
of \cref{rem:pademetzler}, not reflection. The two effects are therefore
distinguished by \emph{when} they occur, and a count taken only at the final time
would rank the schemes backwards.}
\label{stiffRing}%
\end{table}

\enlargethispage{1em}
\begin{figure}[!htbp]  
\centering
\hspace*{-0.3in}
\includegraphics[width=0.7\textwidth]{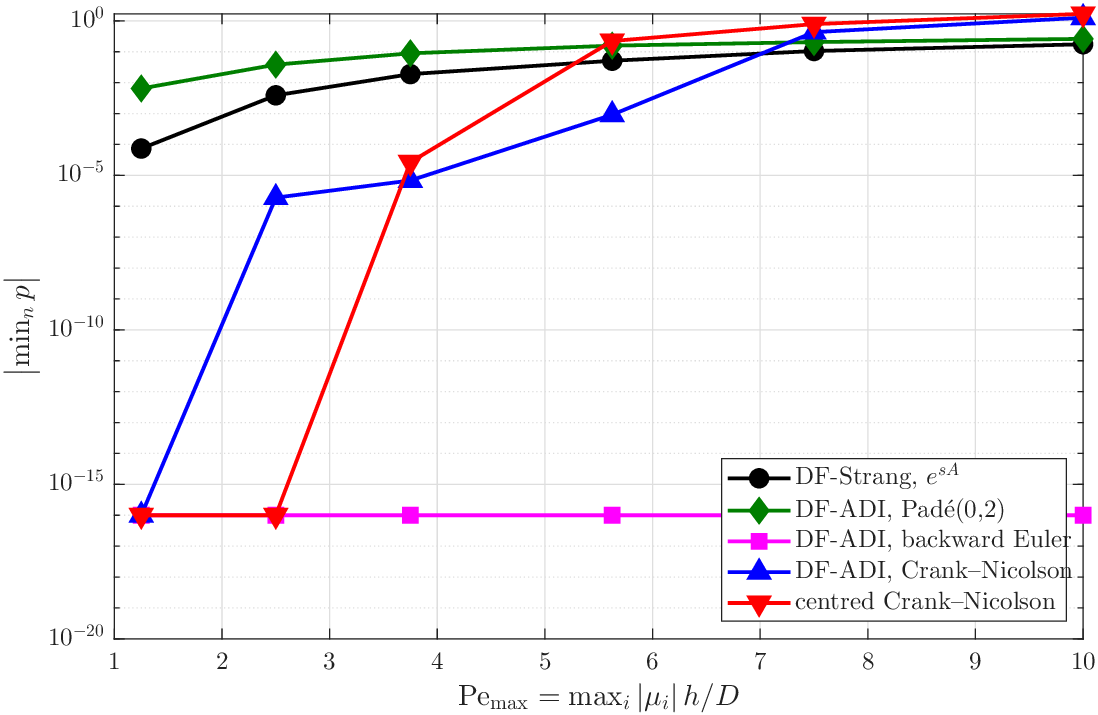}
\vspace{-0.5em}  
\caption{Most negative nodal value attained over the run, $|\min_n \bm p|$, against the maximal cell P\'eclet number $\mathrm{Pe}_{\max}=\max_i|\mu_i|h/D$, the largest ratio of advective to diffusive transport across a single cell on the grid. Data as in \cref{peSweep}: $n=161$ nodes, $\Delta t = 0.01$, $T=0.25$, with $\kappa$ swept to move $\mathrm{Pe}_{\max}$. Values are floored at $10^{-16}$ for plotting. Beyond $\mathrm{Pe}_{\max}\approx5$, where the mesh ceases to resolve the drift, the two trapezoidal realizations separate from the three factors with $r(\infty)=0$ and deteriorate rapidly. The residual undershoot of the exponential and Pad\'e curves at small $\mathrm{Pe}_{\max}$ is the sub-threshold transient of \cref{rem:zenomirror}, since $\Delta t$ here lies below $2\tau_0$ and $2\gamma_r$; it is not ringing.}
\label{Fpeclet}
\end{figure}

\test{9}{ Ringing versus the sub-threshold transient.}
Two distinct phenomena must be separated here. \Cref{peSweep,stiffRing} are
designed to do so, with \cref{Fpeclet} showing the sweep graphically. The first
is temporal ringing, caused by reflection of stiff modes when $r(\infty)<0$; this
is precisely what the L-stability of Pad'e$(0,2)$ is designed to avoid. The
second is the sub-threshold transient of \cref{rem:zenomirror}, in which an
eventually positive factor fails below its own threshold.

In the sweep, $\Delta t=0.01$ is far below the thresholds $2\tau_0$ and
$2\gamma_r$ of \cref{ouThresh}, both of which are $O(0.6)$. The exponential and
Pad'e columns therefore show the sub-threshold transient, which grows with the
P'eclet number, whereas backward Euler remains clean because its threshold
$2\gamma_0$ is two orders of magnitude smaller. None of this is ringing. Ringing
appears only in the last two columns: the trapezoidal factors become negative
already at $\mathrm{Pe}*{\max}=2.5$, at the $10^{-6}$ level, and deteriorate
rapidly beyond $\mathrm{Pe}*{\max}\approx5$. The centred realization becomes
worse than the upwind one from $\mathrm{Pe}_{\max}\approx3.8$ onward.

To isolate ringing, \cref{stiffRing} uses pure diffusion, so the operator is
Metzler and no eventual-positivity threshold is involved. Starting from a
grid-scale spike, we take $\frac{\Delta t}{2}\max_i|L_{ii}|=2.6$, deep in the
stiff regime. The trapezoidal factor reflects the spike and undershoots to
$-6.13$. The exponential and backward Euler factors remain nonnegative to
round-off, while Pad'e$(0,2)$, which also has $r(\infty)=0$, avoids ringing but
still undershoots by $-2.34\cdot10^{-3}$. This residual is neither ringing nor an
eventual-positivity transient: the latter is absent for a Metzler operator. It is
the positivity restriction inherent to second-order rational approximations of
the exponential, even for an M-matrix, and is three orders of magnitude smaller
than the trapezoidal reflection. Backward Euler avoids it only by sacrificing
first-order accuracy. This quantifies the distinction emphasized in
\cref{rem:whyexp,cor:price}: for Crank--Nicolson, the problem is not second-order
accuracy itself but the negative value of $r(\infty)$.

\subsection{One-dimensional jump-diffusion} \label{ssec:numjumps}

The last experiment isolates the jump factor in \eqref{eq:dfadi-jump}. In one
dimension the mixed block is absent, so the composition reduces to
$\Phi_D(\tfrac{\Delta t}{2}),\Phi_J(\Delta t),\Phi_D(\tfrac{\Delta t}{2})$. We
use the Ornstein--Uhlenbeck drift and diffusion of \eqref{eq:ou} with
$\kappa=\sigma=1$, intensity $\lambda=3$, and two standard kernels: Merton, with
$\mu_J=-0.10$, $\sigma_J=0.35$, and Kou, with $q=0.4$, $\eta_1=6$, $\eta_2=8$.

The jump generator is assembled as $J=\lambda(F-\mathcal I)$, where
$F_{ij}=h,\nu(x_i-x_j)$ and the columns are rescaled to sum to one. Thus $F\ge0$
and $\bm1^\top J=0$ exactly. This gives the conservative kernel truncation and
also fixes the sign convention: the generator must be $\lambda(F-\mathcal I)$,
not $\lambda(F+\mathcal I)$.

\test{10}{ The jump factor.} The upper block of \cref{jumpTab} isolates
the jump factor. Since $J$ is Metzler with zero column sums, it is the negative
of an M-matrix; hence $e^{sJ}\ge0$ and $(\mathcal I-sJ)^{-1}\ge0$ for every
$s>0$. Both thresholds therefore vanish, and \cref{cor:rinf} applies directly.

Pad'e$(0,2)$ is different. As noted in \cref{rem:pademetzler}, a Metzler
generator alone does not guarantee positivity of the Pad'e map. Here the
threshold is not only positive but large: $\gamma_r=11.4$ for Merton and $48.5$
for Kou, roughly an order of magnitude above the directional thresholds of
\cref{ouThresh}. The resulting negative entries are small, $-8.5\cdot10^{-4}$ and
$-1.2\cdot10^{-4}$, respectively, and are invisible on resolved data; the
composite runs remain nonnegative to round-off.

The practical conclusion is nevertheless clear. The backward Euler resolvent is
unconditionally positive for $J$ but first order, while Pad'e$(0,2)$ is second
order but requires a threshold that realistic steps may not satisfy. The
exponential is the only factor that is both second order and unconditionally
positive. It is also inexpensive for these jump generators, since
\begin{equation}
e^{sJ}=e^{-s\lambda}\sum_k\frac{(s\lambda)^k}{k!}F^k
\end{equation}
is a Poisson series that can be truncated after only a few terms. Thus, for the
jump block, the exponential is the natural choice—the opposite of the preference
for the directional blocks that motivated this paper.

The lower block confirms the composition accuracy. All three second-order factors
reproduce the exact-in-time reference at rate $2$, while backward Euler converges
at rate $1$. Discrete mass is conserved to round-off by every propagator, as
required by \cref{prop:massadi}, since $\bm1^\top J=0$ together with $\bm1^\top
L_\alpha=0$ satisfies its hypothesis. For the Merton case with the drift switched
off, comparison with the analytic Poisson mixture of Gaussians gives spatial
convergence rates of $2.00$, $2.00$, and $1.98$ on $n=101$ to $801$.
\begin{table}[htbp]
\centering
\scalebox{1.}{
\begin{tabular}{|l|c|c|r|r|}
\toprule
\rowcolor[rgb]{ .792,  .929,  .984}
\textbf{kernel} & $\tau_0$ & $\gamma_0$ & \multicolumn{1}{c|}{$\gamma_r$} &
\multicolumn{1}{c|}{\textbf{worst Pad\'e entry}} \\
\hline
Merton & 0 & 0 & 1.143e+01 & -8.48e-04 \\
\hline
\rowcolor[rgb]{ .557,  .851,  .451}
Kou    & 0 & 0 & 4.845e+01 & -1.21e-04 \\
\bottomrule
\end{tabular}}

\vspace{\baselineskip}

\scalebox{0.95}{
\begin{tabular}{|l|c|c|c|c|}
\toprule
\rowcolor[rgb]{ .792,  .929,  .984}
\textbf{kernel / propagator} & \textbf{exponential} & \textbf{Pad\'e$(0,2)$} &
\textbf{Backward Euler} & \textbf{Crank--Nicolson} \\
\hline
Merton, observed temporal rate & 2.00 & 1.99 & 1.00 & 2.00 \\
\hline
Merton, $\min_n \bm p$         & 1.5e-79 & -7.1e-16 & 1.5e-79 & 1.5e-79 \\
\hline
Merton, mass defect            & 1.7e-14 & 3.7e-14 & 3.1e-15 & 8.1e-15 \\
\hline
Kou, observed temporal rate    & 2.00 & 1.99 & 1.00 & 2.00 \\
\hline
Kou, $\min_n \bm p$            & 1.5e-79 & 1.5e-79 & 1.5e-79 & 1.5e-79 \\
\hline
\rowcolor[rgb]{ .557,  .851,  .451}
Kou, mass defect               & 1.4e-14 & 1.8e-15 & 6.2e-15 & 1.1e-14 \\
\bottomrule
\end{tabular}}
\caption{One-dimensional jump-diffusion, $\lambda=3$, $n=201$, $T=0.5$. Upper
block: thresholds of the jump factor alone. The generator $J=\lambda(F-\mathcal
I)$ is Metzler with $\bm1^\top J=0$, hence the negative of an M-matrix, so
$\tau_0=\gamma_0=0$ exactly. The threshold $\gamma_r$ is not zero, which is
\cref{rem:pademetzler} on a dense generator, and it is an order of magnitude
larger than the directional thresholds of \cref{ouThresh}. Lower block: the
composite $\Phi_D(\tfrac{\Delta t}{2})\Phi_J(\Delta t)\Phi_D(\tfrac{\Delta
t}{2})$ against the exact-in-time reference $e^{T\cdot (L_D+J)}\bm p_0$, refining $N_t$
from $10$ to $160$. Rates are taken at the finest pair, and $\min_n\bm p$ and the
mass defect at $N_t=160$.}
\label{jumpTab}%
\end{table}%

\myparagraph{The two-dimensional composition.}
It remains to check that the jump factor behaves correctly when inserted into the
full step \eqref{eq:dfadi-jump}, where it is adjacent to the mixed factor. These
are the two non-directional factors, so any interaction between them would be the
main new phenomenon introduced by the two-dimensional composition. None is
observed.

On the strong cross-diffusion benchmark with $\lambda=3$ and a product Merton
kernel, the jumps in the two coordinates are independent, so the interaction
being tested is with $A_{xy}$ rather than correlation within the jump measure.
The assembled step converges at rates $2.00$ for the exponential and Pad'e
directional factors and $1.00$ for backward Euler, exactly as in the no-jump
case. The mixed solve is likewise unaffected: the average number of Picard
iterations remains $6.00$, $5.00$, and $5.00$ on $N=48,72,96$, respectively, with
or without the jump factor.

Positivity slightly improves in the presence of jumps. At $N=72$, the minimum
changes from $-7.8\cdot10^{-8}$ to $-4.6\cdot10^{-11}$, and at $N=96$ from
$-1.5\cdot10^{-10}$ to $-3.8\cdot10^{-13}$, consistent with the smoothing effect
of convolution by the jump kernel. The two-dimensional jump generator also gives
the same threshold structure as in one dimension: $\tau_0=\gamma_0=0$ and
$\gamma_r=1.5\cdot10^{1}$. Thus, the extension to two dimensions introduces no
new interaction or positivity restriction, and we record the result without a
separate table.

\section{Discussion and Conclusions} \label{sec:conclusion}

This paper introduced a positivity-preserving alternating direction implicit (ADI) scheme for anisotropic Fokker--Planck equations, obtained by replacing the Krylov-computed matrix exponentials of the companion Diagonal Frog (DF) scheme \cite{ItkinDF2026} with rational maps of the discrete generators. The substitution changes the computational cost per step from $O(m^2N + m^3)$ to $O(N)$, with $m$ the Krylov dimension, while preserving the spatial discretization, the conservative closure, the implicit mixed factor with its factorized Picard solver, and the symmetric composite structure of the original step.

The central theoretical result that enables this replacement is the $r(\infty)$ criterion. For the class of EM generators produced by the DF discretization, the entrywise sign of a rational map $r(\gamma L)$ in the limit of large step sizes is governed by the single number $r(\infty) = \lim_{|z|\to\infty} r(z)$. Maps with $0 \le r(\infty) < 1$ possess a genuine one-sided positivity window $\gamma \ge \gamma_r$, a mirror of the eventual-positivity threshold $\tau_0$ of the exponential. Maps with $r(\infty) < 0$, most notably the trapezoidal (Crank--Nicolson) map with $r(\infty) = -1$, admit at most a bounded two-sided window. Our numerical experiments find this window to be empty on every advection-dominated mesh tested, and, in the one weakly eventually positive case where it does open, too narrow by orders of magnitude to be usable. This explains why Crank--Nicolson cannot carry a per-factor positivity guarantee despite remaining positive in practice on well-resolved meshes. The exponential itself corresponds to the case $r(\infty)=0$, and the subdiagonal Pad\'e$(0,2)$ approximant, which is second-order, L-stable and also satisfies $r(\infty)=0$, inherits the same one-sided window structure. The scheme built on it, DF-ADI, is therefore the first ADI scheme for this class of equations whose directional factors are individually and provably nonnegative above an explicit, computable threshold.

A second structural contribution is the clarification of the relationship to classical stabilizing-correction ADI schemes (Douglas, MCS, Hundsdorfer--Verwer). We have shown that the $r(\infty)$ criterion applies to steps that are multiplicative in the generators, as is the DF-ADI step. The affine stage structure of the stabilizing-correction family, which adds and subtracts stage vectors, admits no per-factor sign control regardless of whether the mixed derivative is treated implicitly or explicitly. Promoting the cross term to a third implicit direction inside an HV or MCS scheme preserves second-order accuracy and exact mass conservation but does not alter the signed operands at each stage, and no positivity guarantee becomes available. The dividing line is therefore multiplicative against affine, and the DF-ADI scheme sits on the multiplicative side with the per-factor guarantees that entails.

The numerical experiments confirm the theoretical picture. On the Ornstein--Uhlenbeck benchmark the Pad\'e$(0,2)$ factor converges at rate $1.98$ in time, matching the trapezoidal factor and confirming that the second-order accuracy of the flagship is not paid for anywhere. The admissible step region $[2\gamma_r, \Theta]$ has the same shape as the exponential scheme's $[2\tau_0, \Theta]$, with $\gamma_r$ of the same order as $\tau_0$ and, unlike the backward Euler threshold $\gamma_0$, not shrinking under mesh refinement. On the anisotropic Gaussian benchmark at strong cross-diffusion ($\rho=0.95$) the three multiplicative second-order propagators are indistinguishable to three digits at every refinement level, converge at rate $2.11$, and undershoot by $5\cdot10^{-8}$, six orders of magnitude below the peak density, and backward Euler by $6\cdot10^{-7}$. The stabilizing-correction comparators, which treat the cross term explicitly, leave their stability region on the same joint-refinement path and undershoot by $4.8\cdot10^{-2}$ and by $2.3\cdot10^{12}$ respectively. The wall-time advantage of the ADI over the exponential DF scheme grows with the mesh, from a factor of $10$ at $N=48$ per direction to $32$ at $N=240$, because the Krylov dimension of the exponential grows with the mesh while the banded cost of the ADI does not.

The scheme extends by transpose duality to the backward Kolmogorov equation, where the positivity windows transfer unchanged, and to jump-diffusion models through the addition of a nonnegative jump factor. The resulting architecture offers a complete, second-order, provably conservative and conditionally positivity-preserving solver for the Fokker--Planck and backward Kolmogorov equations on tensor grids.

The positivity mechanism here differs in kind from the nonlinear limiters used for hyperbolic problems. Flux-corrected transport and ENO/WENO schemes enforce nonnegativity by making the spatial operator solution dependent, clipping or reconstructing fluxes at each step. The DF-ADI factors are linear and fixed. Nonnegativity is a spectral property of the rational map applied to a fixed generator, certified once through the $r(\infty)$ criterion, rather than imposed at run time. The two approaches are complementary: a limiter can be layered on for the under-resolved regime where the guaranteed window is empty, at the cost of the exact linearity, though not of the mass conservation of \cref{prop:massadi}, which follows from the conservative flux form and survives a flux-form limiter unchanged.

The limiter envisaged in the preceding paragraph is developed in \cite{ItkinFCDF2026}, the third paper of this series, where positivity is secured through the sole recourse that Godunov's theorem leaves open, a nonlinear discretization. The second-order directional operator is separated into a monotone M-matrix core and an antidiffusive flux correction, and a Zalesak-type limiter is applied within the implicit banded solve. The resulting Flux-Corrected DF scheme is unconditionally positive at every step size and second-order wherever the density is resolved. Because the limiter acts on the interface fluxes rather than on nodal values, exact discrete mass conservation is retained for every value of the limiter. The scheme carries no lower step threshold, and it therefore governs precisely the small-step regime in which the linear windows of the present paper and of \cite{ItkinDF2026} lose positivity.

The two constructions are complementary across the range of step sizes. The resolvent and Pad\'e$(0,2)$ windows established here govern the large steps, the flux-limited scheme governs the small steps, and \cite{ItkinFCDF2026} shows that in one dimension their combination yields a positive and exactly conservative realization for every step size. The same pairing is expected to close the anisotropic case. A flux-limited directional factor is positive without a lower step bound, and it therefore replaces the interval $[2\gamma_r,\Theta]$ of \cref{prop:posadi} with one that extends to arbitrarily small steps. The admissible region of \cref{crossWindow} is empty only because its lower endpoint $2\gamma_r$ exceeds the parabolic upper endpoint $\Theta$. Once the limiter is transported into the directional factors of the mixed-derivative scheme, that obstruction is removed, and the region is expected to become nonempty.

Several directions remain open. The $r(\infty)$ criterion is stated here for the class of EM generators produced by the DF discretization. Extending it to broader classes of eventually positive or eventually nonnegative generators, including those arising from non-conservative boundary closures or non-divergence-form discretizations, would widen its applicability. The relationship between the order of the Pad\'e approximant and the sharpness of the positivity threshold is not yet understood quantitatively, and the a priori bounds we provide for $\gamma_r$ are conservative. Tighter spectral estimates would make the criterion a practical design tool at run time. Finally, the multiplicative structure of the DF-ADI step makes it a natural candidate for adaptive step-size control and for extension to three spatial dimensions. The $r(\infty)$ criterion itself carries over to three dimensions without change, since the directional factors remain one-dimensional maps of one-dimensional generators; only the mixed-block solver requires a further level of factorization of the factorized Picard iteration (but see \cite{Itkin3D,Itkin2018Stochastic} as an example).

\section*{Disclosure statement}

No potential conflict of interest was reported by the authors.

\section*{Funding}

No funding was received.

\section*{Disclaimer}

Opinions expressed here are author's own, and do not represent views of their employers. A standard disclaimer applies.


\printbibliography

\appendix
\label{app:proofs}
\appendixpage
\numberwithin{equation}{section}
\setcounter{equation}{0}

\section{Proofs of various theorems}

Throughout this appendix $L\in\R^{n\times n}$ denotes the generator fixed in
\cref{sec:resolvent}. It is irreducible, has zero column sums, spectral abscissa
$0$ with the eigenvalue $0$ simple, and is eventually exponentially positive. We
write $\bm1$ for the all-ones vector and $\bm p^\infty>0$ for the stationary
density, normalized by $\bm1^\top\bm p^\infty=1$, so that $L\bm p^\infty=0$ and
$\bm1^\top L=0$. The spectral projector onto $\ker L$ along $\operatorname{ran}L$
is the rank-one matrix $\Pi=\bm p^\infty\bm1^\top$, with entries $\Pi_{ij}=p^\infty_i$,
and it satisfies $\Pi^2=\Pi$ and $L\Pi=\Pi L=0$.

Because $0$ is simple and the abscissa is $0$, every other eigenvalue of $L$ has
strictly negative real part. Define the spectral gap
\begin{equation}\label{eq:gap}
\delta:=\min\{-\Ree\lambda:\lambda\in\operatorname{spec}(L),\ \lambda\ne0\}>0,
\end{equation}
and let $L_0$ be the restriction of $L$ to its invariant subspace
$\operatorname{ran}L=\operatorname{ran}(\mathcal I-\Pi)$. Then $L_0$ is invertible
and $\operatorname{spec}(L_0)=\operatorname{spec}(L)\setminus\{0\}\subset\{\Ree z\le-\delta\}$.
The subspaces $\ker L$ and $\operatorname{ran}L$ are complementary and
$L$-invariant, so for any function $f$ analytic on a neighbourhood of
$\operatorname{spec}(L)$ the primary matrix function $f(L)$ respects the splitting,
\begin{equation}\label{eq:blocksplit}
f(L)\,\Pi=f(0)\,\Pi,\qquad f(L)\,(\mathcal I-\Pi)=f(L_0)\,(\mathcal I-\Pi),
\end{equation}
where $f(L_0)$ is understood to act on $\operatorname{ran}L$. The proofs below use
\eqref{eq:blocksplit} with $f$ the resolvent and with $f$ the stability function.

\subsection{Proof of \cref{thm:resolvent}} \label{app:resolvent}

\begin{proof}
\emph{The window \eqref{eq:reswindow}.} Fix $\gamma>0$ and put $s=1/\gamma$. From
$\mathcal I-\gamma L=s^{-1}(s\mathcal I-L)$ we obtain
$(\mathcal I-\gamma L)^{-1}=s\,(s\mathcal I-L)^{-1}$. Apply \eqref{eq:blocksplit} to
the resolvent $f(z)=(s-z)^{-1}$, which is analytic at every $z\ne s$ and hence on
$\operatorname{spec}(L)$ for $s>0$. On $\ker L$ it gives
$(s\mathcal I-L)^{-1}\Pi=s^{-1}\Pi$. On $\operatorname{ran}L$ it gives the reduced
resolvent
\begin{equation}\label{eq:redres}
R(s):=(s\mathcal I-L)^{-1}(\mathcal I-\Pi)=(s\mathcal I-L_0)^{-1}(\mathcal I-\Pi).
\end{equation}
The map $s\mapsto(s\mathcal I-L_0)^{-1}$ is analytic wherever
$s\notin\operatorname{spec}(L_0)$. By \eqref{eq:gap} the distance from $0$ to
$\operatorname{spec}(L_0)$ is at least $\delta$, so this map, and therefore $R$, is
analytic on the disc $|s|<\delta$ and continuous on the compact segment
$[0,\delta/2]$. Consequently
\begin{equation}\label{eq:Cstar}
C_\star:=\sup_{s\in[0,\delta/2]}\|R(s)\|_{\max}<\infty.
\end{equation}
This is the point at which the spectral gap enters. The only singularity of
$(s\mathcal I-L)^{-1}$ at $s=0$ is the rank-one term $s^{-1}\Pi$, while the
complementary part $R(s)$ stays bounded because $L_0$ is invertible with spectrum
at distance at least $\delta$ from the origin.

Assembling the two pieces,
\begin{equation}\label{eq:resassemble}
(\mathcal I-\gamma L)^{-1}=s\bigl(s^{-1}\Pi+R(s)\bigr)=\Pi+s\,R(s),
\end{equation}
so entrywise $[(\mathcal I-\gamma L)^{-1}]_{ij}=p^\infty_i+s\,R_{ij}(s)$. For
$s\le\delta/2$ the bound \eqref{eq:Cstar} gives $|s\,R_{ij}(s)|\le sC_\star$, whence
\begin{equation}\label{eq:resbound-proof}
[(\mathcal I-\gamma L)^{-1}]_{ij}\ \ge\ \min_kp^\infty_k-sC_\star.
\end{equation}
The right-hand side is strictly positive once $s<\min_kp^\infty_k/C_\star$. Both
conditions on $s$ hold as soon as
$\gamma\ge\gamma_0^{\mathrm{bd}}$ of \eqref{eq:resolvent-bound}, which is
finite; hence $\gamma_0\le\gamma_0^{\mathrm{bd}}$, the symbol $\gamma_0$ being
reserved throughout for the sharp threshold. This proves \eqref{eq:reswindow}, with every entry, the diagonal included,
strictly positive because $\min_kp^\infty_k>0$. Taken over $s\in(0,\delta/2]$, the
same estimate \eqref{eq:resbound-proof} is what \cref{ssec:quantitative} turns into
the a priori bound \eqref{eq:resolvent-bound}.

\emph{Failure for small $\gamma$.} Suppose $L_{ij}<0$ for some $i\ne j$, and let
$\|\cdot\|_2$ be the spectral norm. For $\gamma\|L\|_2<1$ the Neumann series
converges and
\begin{equation}\label{eq:neumann}
(\mathcal I-\gamma L)^{-1}=\mathcal I+\gamma L+\gamma^2L^2(\mathcal I-\gamma L)^{-1}.
\end{equation}
For $\gamma\le1/(2\|L\|_2)$ one has
$\|(\mathcal I-\gamma L)^{-1}\|_2\le(1-\gamma\|L\|_2)^{-1}\le2$, so the remainder
$E(\gamma):=\gamma^2L^2(\mathcal I-\gamma L)^{-1}$ obeys
$\|E(\gamma)\|_{\max}\le\|E(\gamma)\|_2\le2\gamma^2\|L\|_2^2$. Hence
\begin{equation*}
[(\mathcal I-\gamma L)^{-1}]_{ij}=\gamma L_{ij}+E_{ij}(\gamma),\qquad
|E_{ij}(\gamma)|\le2\gamma^2\|L\|_2^2 .
\end{equation*}
For $0<\gamma<\gamma_1:=|L_{ij}|/(2\|L\|_2^2)$ the linear term dominates in modulus
and the entry is negative. Since $|L_{ij}|\le\|L\|_2$ we have
$\gamma_1\le1/(2\|L\|_2)$, so \eqref{eq:neumann} is valid on this range. Positivity
therefore cannot reach $\gamma=0$, and $\gamma_0\ge\gamma_1>0$.

\emph{Metzler case.} If $L$ has no negative off-diagonal entry, then for every
$\gamma>0$ the matrix $\mathcal I-\gamma L$ is a $Z$-matrix, its off-diagonal
entries being $-\gamma L_{ij}\le0$. Its eigenvalues are $1-\gamma\lambda$ with
$\lambda\in\operatorname{spec}(L)$, and since $\Ree\lambda\le0$ each has real part
$1-\gamma\,\Ree\lambda\ge1>0$. A $Z$-matrix all of whose eigenvalues have positive
real part is a nonsingular M-matrix, so $(\mathcal I-\gamma L)^{-1}\ge0$;
irreducibility of $L$ makes $\mathcal I-\gamma L$ an irreducible nonsingular
M-matrix, whose inverse is entrywise strictly positive. Thus $\gamma_0=0$ and
\eqref{eq:reswindow} holds for every $\gamma>0$.
\end{proof}

\subsection{Proof of \cref{thm:limit}}  \label{app:limit}

\begin{proof}
Write $r=P/Q$ with $P,Q$ coprime and $Q(0)\ne0$, so that $r(0)=1$ is finite. Let
$z_1,\dots,z_q$ be the poles of $r$, that is the roots of $Q$, and set
$\rho:=\max_l|z_l|$. A-acceptability places every pole in the open right
half-plane, so for $\gamma>0$ the eigenvalues of $\gamma L$, namely $0$ and
$\{\gamma\lambda:\lambda\in\operatorname{spec}(L_0)\}$, all lie in the closed left
half-plane and avoid the poles. Thus $r(\gamma L)$ is defined for every $\gamma>0$.

By \eqref{eq:blocksplit} with $f=r$, on $\ker L$ the map contributes
$r(0)\,\Pi=\Pi$ for every $\gamma$, and it remains to evaluate $r(\gamma L_0)$ on
$\operatorname{ran}L$ as $\gamma\to\infty$. Fix a Jordan factorization
$L_0=S\bigl(\bigoplus_bJ_b\bigr)S^{-1}$ with $S$ independent of $\gamma$, where the
block $J_b=\lambda_bI_{m_b}+N_b$ has eigenvalue $\lambda_b\in\operatorname{spec}(L_0)$
and nilpotent shift $N_b$ with $N_b^{m_b}=0$. For a single block the
primary-function formula reads
\begin{equation}\label{eq:jordanfun}
r(\gamma J_b)=\sum_{k=0}^{m_b-1}\frac{\gamma^k\,r^{(k)}(\gamma\lambda_b)}{k!}\,N_b^{\,k}.
\end{equation}
The term $k=0$ is $r(\gamma\lambda_b)I_{m_b}$. Because $r$ is rational with the
finite value $r(\infty)$, it admits the Laurent expansion
$r(z)=r(\infty)+\sum_{j\ge1}c_jz^{-j}$ for $|z|>\rho$, and termwise
differentiation, valid on that annulus, gives
\begin{equation}\label{eq:derdecay}
r^{(k)}(z)=O\bigl(|z|^{-(k+1)}\bigr),\qquad|z|\to\infty .
\end{equation}
Since $|\gamma\lambda_b|\ge\gamma\delta$ by \eqref{eq:gap}, once $\gamma>\rho/\delta$
every eigenvalue satisfies $|\gamma\lambda_b|>\rho$ and \eqref{eq:derdecay} applies.
The term $k=0$ therefore tends to $r(\infty)I_{m_b}$, and for $1\le k\le m_b-1$,
\begin{equation*}
\Bigl\|\tfrac{\gamma^k}{k!}\,r^{(k)}(\gamma\lambda_b)\,N_b^{\,k}\Bigr\|
=O\bigl(\gamma^{k}\cdot\gamma^{-(k+1)}\bigr)=O(\gamma^{-1})\longrightarrow0 .
\end{equation*}
This is the quantitative content behind the vanishing of the Jordan
contributions. A-acceptability by itself only bounds $r$ on the left half-plane.
It is the decay rate \eqref{eq:derdecay}, inherited from finiteness of $r(\infty)$,
that overcomes the growth $\gamma^k$ carried by the nilpotent coefficients. Hence
$r(\gamma J_b)\to r(\infty)I_{m_b}$ for each block, and as $S$ is fixed,
$r(\gamma L_0)\to r(\infty)\mathcal I$ on $\operatorname{ran}L$. Combining the two
subspaces through \eqref{eq:blocksplit},
\begin{equation*}
\lim_{\gamma\to\infty}r(\gamma L)=\Pi+r(\infty)(\mathcal I-\Pi),
\end{equation*}
which is \eqref{eq:limitmatrix}. Using $\Pi_{ij}=p^\infty_i$, the off-diagonal
entries of the limit are $(1-r(\infty))p^\infty_i$ and the diagonal entries are
$p^\infty_i+r(\infty)(1-p^\infty_i)$.
\end{proof}

\subsection{Proof of \cref{cor:rinf}}  \label{app:rinf}

\begin{proof}
Write $M_\infty=\Pi+r(\infty)(\mathcal I-\Pi)$ for the limit \eqref{eq:limitmatrix},
and recall $r(\gamma L)\to M_\infty$ entrywise.

(a) If $0\le r(\infty)<1$ then $1-r(\infty)>0$. Every off-diagonal entry
$(1-r(\infty))p^\infty_i$ of $M_\infty$ is then positive, and every diagonal entry
$(1-r(\infty))p^\infty_i+r(\infty)$ is the sum of a positive number and
$r(\infty)\ge0$, hence positive. Put $\mu:=\min_{i,j}(M_\infty)_{ij}>0$. Choose
$\gamma_r$ with $\|r(\gamma L)-M_\infty\|_{\max}<\mu$ for all $\gamma\ge\gamma_r$.
Then each entry of $r(\gamma L)$ is at least
$(M_\infty)_{ij}-\|r(\gamma L)-M_\infty\|_{\max}>(M_\infty)_{ij}-\mu\ge0$, so
$r(\gamma L)>0$ on $[\gamma_r,\infty)$.

(b) If $r(\infty)<0$, the diagonal entry
$(M_\infty)_{ii}=p^\infty_i+r(\infty)(1-p^\infty_i)$ is negative precisely when
$p^\infty_i(1+|r(\infty)|)<|r(\infty)|$, that is
$p^\infty_i<|r(\infty)|/(1+|r(\infty)|)$. At a node $i$ with this property
$(M_\infty)_{ii}=-\eta$ for some $\eta>0$, and entrywise convergence gives
$[r(\gamma L)]_{ii}<-\eta/2<0$ for all $\gamma$ past some $\gamma_+$. Thus
$r(\gamma L)$ has a negative diagonal entry for all large $\gamma$, and
$\{\gamma>0:r(\gamma L)\ge0\}$ is bounded above.

(c) The hypothesis $r'(0)=1$ holds for every method used here, since backward
Euler, the trapezoidal rule and Pad\'e$(0,2)$ all satisfy $r(z)=1+z+O(z^2)$.
Suppose $L_{ij}<0$ for some $i\ne j$. As $Q(0)\ne0$, $r$ is analytic at $0$ with
power series $r(z)=1+z+\sum_{k\ge2}a_kz^k$ of radius $\rho_0=\min_l|z_l|>0$. For
$\gamma\|L\|_2<\rho_0$ the series $r(\gamma L)=\sum_{k\ge0}a_k(\gamma L)^k$
converges, giving
\begin{equation*}
r(\gamma L)=\mathcal I+\gamma L+\gamma^2G(\gamma),\qquad
G(\gamma):=\sum_{k\ge2}a_k\gamma^{k-2}L^k,
\end{equation*}
with $\|G(\gamma)\|_{\max}$ bounded for $\gamma$ near $0$. The $(i,j)$ entry equals
$\gamma L_{ij}+O(\gamma^2)$, which is negative for all sufficiently small
$\gamma>0$. Together with part (b), every nonnegativity window is confined to a
compact subinterval $[\gamma_-,\gamma_+]\subset(0,\infty)$.
\end{proof}

\subsection{Proof of \cref{thm:cayley}}  \label{app:cayley}

\begin{proof}
Recall $r(z)=(1+z/2)/(1-z/2)$, so $r(\infty)=-1$, and abbreviate the
backward-Euler resolvent at step $\gamma/2$ by
\begin{equation}\label{eq:Gdef}
G(\gamma):=\Bigl(\mathcal I-\tfrac{\gamma}{2}L\Bigr)^{-1},
\end{equation}
so that the identity \eqref{eq:cayleyident} reads $r(\gamma L)=-\mathcal I+2\,G(\gamma)$.

\emph{Part (a).} Setting $s'=2/\gamma$ in \eqref{eq:resassemble} gives
$G(\gamma)=\Pi+s'R(s')\to\Pi$ as $\gamma\to\infty$, so $r(\gamma L)\to-\mathcal I+2\Pi$
entrywise. This is \eqref{eq:limitmatrix} with $r(\infty)=-1$. The diagonal limit
at node $i$ is $-1+2p^\infty_i=2p^\infty_i-1$. If $p^\infty_i<1/2$ this equals
$-\eta$ with $\eta>0$, and entrywise convergence yields $\gamma_+<\infty$ with
$[r(\gamma L)]_{ii}<-\eta/2<0$ for all $\gamma>\gamma_+$.

\emph{Part (b).} The term $-\mathcal I$ in \eqref{eq:cayleyident} touches only the
diagonal, so for $i\ne j$ one has $[r(\gamma L)]_{ij}=2\,G_{ij}(\gamma)$. If
$L_{ij}<0$, the small-$\gamma$ failure proved in \cref{app:resolvent}, with
$\gamma$ there replaced by $\gamma/2$, makes $G_{ij}(\gamma)<0$ for all
sufficiently small $\gamma$. Hence $[r(\gamma L)]_{ij}<0$ on $(0,\gamma_-)$ for some
$\gamma_->0$.

\emph{Part (c).} We prove boundedness of the window first, then the nonemptiness
criterion.

The nodes carry a probability vector, so $\sum_ip^\infty_i=1$ over $n\ge2$ nodes.
Unless $n=2$ with $p^\infty=(\tfrac12,\tfrac12)$, some node has $p^\infty_i<1/2$;
on any mesh with $n\ge3$ this is automatic, since
$\min_ip^\infty_i\le1/n<1/2$. Fix such a node. Part (a) then supplies
$\gamma_+<\infty$ with $r(\gamma L)\not\ge0$ for $\gamma>\gamma_+$, and part (b)
supplies $\gamma_->0$ with $r(\gamma L)\not\ge0$ for $\gamma<\gamma_-$. Therefore
\begin{equation*}
\{\gamma>0:r(\gamma L)\ge0\}\subseteq[\gamma_-,\gamma_+],
\end{equation*}
a bounded window, possibly empty.

For nonemptiness, read the sign of $r(\gamma L)=-\mathcal I+2G(\gamma)$ entry by
entry. Off the diagonal, $[r(\gamma L)]_{ij}=2G_{ij}(\gamma)$ is nonnegative if and
only if $G_{ij}(\gamma)\ge0$. On the diagonal,
$[r(\gamma L)]_{ii}=-1+2G_{ii}(\gamma)$ is nonnegative if and only if
$G_{ii}(\gamma)\ge\tfrac12$. Introduce
\begin{equation*}
W_{\mathrm{res}}:=\{\gamma>0:G(\gamma)>0\ \text{entrywise}\},\qquad
W_{\mathrm{diag}}:=\{\gamma>0:G_{ii}(\gamma)\ge\tfrac12\ \text{for all }i\}.
\end{equation*}
Here $W_{\mathrm{res}}$ is the resolvent window of \cref{thm:resolvent} at shift
$\gamma/2$, so $W_{\mathrm{res}}\supseteq[2\gamma_0,\infty)$ by that theorem. Let
$\gamma\in W_{\mathrm{res}}\cap W_{\mathrm{diag}}$. Then every off-diagonal entry of
$G(\gamma)$ is positive, so every off-diagonal entry of $r(\gamma L)$ is positive,
and every diagonal entry of $G(\gamma)$ is at least $\tfrac12$, so every diagonal
entry of $r(\gamma L)$ is nonnegative. Hence $r(\gamma L)\ge0$, and $\gamma$ lies in
the window. Consequently, whenever $W_{\mathrm{res}}\cap W_{\mathrm{diag}}\ne\emptyset$
the window $[\gamma_-,\gamma_+]$ is nonempty. This is the criterion stated in
part (c).

The two constraints pull in opposite directions. As $\gamma\to\infty$,
$G_{ii}(\gamma)\to p^\infty_i<\tfrac12$, so $W_{\mathrm{diag}}$ excludes all large
steps, whereas $W_{\mathrm{res}}$ excludes all small steps in the EM case by
part (b). Overlap requires the two mechanisms to meet, and \cref{rem:cnempty}
records that on every mesh tested they do not.
\end{proof}

\subsection{Proof of \cref{thm:pade}}  \label{app:pade}

\begin{proof}
The positivity claim \eqref{eq:padewindow} is \cref{cor:rinf}(a) with $r_{02}(\infty)=0$, whose limit is $\Pi>0$; hence $r_{02}(\gamma L)>0$ for all
$\gamma\ge\gamma_r$. The small-$\gamma$ failure is \cref{cor:rinf}(c): since
$r_{02}(z)=1+z+O(z^2)$ has $r_{02}'(0)=1$, a negative off-diagonal entry of $L$
produces a negative entry of $r_{02}(\gamma L)$ for all sufficiently small
$\gamma>0$, so $\gamma_r>0$ in the EM case proper.
\end{proof}

\subsection{Proof of \cref{prop:order}} \label{app:order}

\begin{proof}
For the directional factors, the Taylor expansion of the Pad\'e$(0,2)$
approximant gives
\begin{equation}
\left(\mathcal I-sA_\alpha+\frac{s^2}{2}A_\alpha^2\right)^{-1} = \mathcal I+sA_\alpha+\frac{s^2}{2}A_\alpha^2+O(s^3), \qquad \alpha\in\{x,y\}.
\end{equation}
Since
\begin{equation}
e^{sA_\alpha} = \mathcal I+sA_\alpha+\frac{s^2}{2}A_\alpha^2+O(s^3),
\end{equation}
it follows that
\begin{equation} \label{eq:app-dir-error}
\Phi_\alpha(s) = e^{sA_\alpha}+O(s^3).
\end{equation}
Thus, for the directional half-steps,
\begin{equation}
\Phi_\alpha\!\left(\frac{\Delta t}{2}\right) = e^{\frac{\Delta t}{2}A_\alpha} +O(\Delta t^3).
\end{equation}

For the central factor, the trapezoidal expansion gives
\begin{equation}
\left(\mathcal I-\frac{\Delta t}{2}A_{xy}\right)^{-1} \left(\mathcal I+\frac{\Delta t}{2}A_{xy}\right) = \mathcal I+\Delta t A_{xy} + \frac{\Delta t^2}{2}A_{xy}^2 + O(\Delta t^3),
\end{equation}
whereas
\begin{equation}
e^{\Delta t A_{xy}} = \mathcal I+\Delta t A_{xy} + \frac{\Delta t^2}{2}A_{xy}^2 + O(\Delta t^3).
\end{equation}
Hence
\begin{equation} \label{eq:app-central-error}
\Phi_{xy}(\Delta t) = e^{\Delta t A_{xy}} + O(\Delta t^3).
\end{equation}
Therefore every factor in \eqref{eq:dfadi} differs from its corresponding
exponential by $O(\Delta t^3)$.

Let
\begin{equation}
S(\Delta t) = e^{\frac{\Delta t}{2}A_x} e^{\frac{\Delta t}{2}A_y} e^{\Delta t A_{xy}}
e^{\frac{\Delta t}{2}A_y} e^{\frac{\Delta t}{2}A_x}
\end{equation}
denote the corresponding symmetric exponential composition. Since the number of
factors is fixed and the factors remain uniformly bounded for sufficiently small
$\Delta t$, replacing the five exponential factors by their rational
approximations changes the product by only $O(\Delta t^3)$. Thus
\begin{equation}
\Phi_x\!\left(\frac{\Delta t}{2}\right)  \Phi_y\!\left(\frac{\Delta t}{2}\right)
 \Phi_{xy}(\Delta t)  \Phi_y\!\left(\frac{\Delta t}{2}\right)  \Phi_x\!\left(\frac{\Delta t}{2}\right) = S(\Delta t) + O(\Delta t^3).
\end{equation}

The product $S(\Delta t)$ is the symmetric Strang composition of the directional
and mixed generators. Hence, without requiring the individual operators to
commute,
\begin{equation}
S(\Delta t) = e^{\Delta t(A_x+A_y+A_{xy})} + O(\Delta t^3).
\end{equation}
Consequently, the DF--ADI step has local temporal error $O(\Delta t^3)$ and
therefore global temporal error $O(\Delta t^2)$.

For the spatial discretization, the directional discretizations are second-order
accurate, while Proposition~10 in \cite{ItkinDF2026} (P10) establishes that the
second-order coupling used in the central factor has spatial defect
\begin{equation}
O\!\left(\max(h_x^2,h_y^2)\right).
\end{equation}
Therefore, under the assumptions of P10, in particular the log-Lipschitz condition
(that the discrete density is log-Lipschitz on the stencil required for its positivity claim), the spatial defect of the assembled scheme is
\begin{equation}
O\!\left(\max(h_x^2,h_y^2)\right).
\end{equation}
Combining the temporal and spatial errors, and assuming the resulting scheme is stable, gives
\begin{equation}
\|\bm p(t_n)-\bm p^{\,n}\| \leq C\left[ \Delta t^2+ \max(h_x^2,h_y^2) \right].
\end{equation}
Hence, under the joint refinement $\Delta t\sim h$, with $h=\max(h_x,h_y)$, both
contributions are $O(h^2)$, and the assembled DF--ADI scheme is second-order
accurate in both time and space.
\end{proof}

\subsection{Proof of \cref{prop:massadi}}  \label{app:massadi}

\begin{proof}
This is immediate from the mass corollary, applied factor by factor, since every factor of the scheme is a rational function $r$ of a zero-column-sum generator with $r(0)=1$.
\end{proof}

\subsection{Proof of \cref{prop:posadi}}  \label{app:posadi}

\begin{proof}
The directional Pad\'e factors are nonnegative for $\tfrac{\Delta t}{2}\ge\gamma_r^{(\alpha)}$ due to the Pad\'e windows, which lift to the Kronecker structure with the same thresholds. The central factor is nonnegative for $\Delta t\le\Theta$. Therefore, within the intersection of these windows, every factor is entrywise nonnegative, rendering the composite propagator nonnegative and column-stochastic.

For any $p\ge 0$ the nonnegativity of $S_{\mathrm{ADI}}(\Delta t)$ gives
$S_{\mathrm{ADI}}(\Delta t)\,p\ge 0$, and the unit column sums give
\begin{equation}
\big\|S_{\mathrm{ADI}}(\Delta t)\,p\big\|_1
= \mathbf{1}^\top S_{\mathrm{ADI}}(\Delta t)\,p
= \mathbf{1}^\top p
= \|p\|_1.
\end{equation}
Hence $S_{\mathrm{ADI}}(\Delta t)$ is an $\ell_1$ isometry on the nonnegative
cone, and in particular $\ell_1$-nonexpansive there.

\end{proof}

\end{document}